\documentclass[1p]{elsarticle}
\usepackage[english]{babel}
\usepackage[utf8]{inputenc}

\usepackage[colorlinks=true,citecolor=black,linkcolor=black,urlcolor=black]{hyperref}	
\usepackage{amsthm}
\usepackage{amssymb,amsmath,yfonts}
\usepackage[usenames, dvipsnames]{xcolor}
\usepackage{tikz}
\usetikzlibrary{calc}
\usetikzlibrary{decorations.pathmorphing}
\usetikzlibrary{shapes,snakes}
\usetikzlibrary{decorations.markings}
\usetikzlibrary{3d,backgrounds,intersections}
\usetikzlibrary{arrows,arrows.meta,chains,matrix,scopes}
\usepackage{pgfplots}

\usepackage[linesnumbered,lined,boxed,commentsnumbered]{algorithm2e}

\usepackage[color=green!30]{todonotes}

\usepackage{faktor}
\usepackage{subfig}
\usepackage{stmaryrd}
\usepackage{float}

\usepackage[title]{appendix}

\usepackage{xspace}

\newcommand{\tikzstylemacro}{
\tikzstyle{n} = [draw,circle,minimum size=0.8mm,inner sep=0pt,outer sep=0pt,fill=black];
\tikzstyle{ns} = [draw,rectangle,minimum size=1.8mm,inner sep=0pt,outer sep=0pt,fill=black];
\tikzstyle{v} = [draw,circle,minimum size=5.0mm,inner sep=1pt,outer sep=0pt];
\tikzstyle{b} = [line width=1pt,color=black];
\tikzstyle{g} = [line width=1pt,color=NavyBlue];
\tikzstyle{r} = [line width=1pt, color=red,dashdotted];
\tikzstyle{rr} = [line width=1pt, color=red];
}

\newcommand{\KIVMixed}{K_4^{mixed}}

\newtheorem{theorem}{Theorem}
\newtheorem{lemma}[theorem]{Lemma}
\newtheorem{claim}[theorem]{Claim}
\newtheorem{obs}[theorem]{Observation}
\newtheorem{question}[theorem]{Question}

\newtheorem{proposition}[theorem]{Proposition}

\newtheorem{corollary}[theorem]{Corollary}
\newtheorem{conj}[theorem]{Conjecture}

\newdefinition{definition}[theorem]{Definition}
\newdefinition{rem}[theorem]{Remark}
\newdefinition{notation}[theorem]{Notation}

\newcommand{\set}[1]{\left\{ #1 \right\}}
\newcommand{\setk}[1]{\left\llbracket #1 \right\rrbracket}
\newcommand{\abs}[1]{\left| #1 \right|}
\newcommand{\setcond}[2]{\left\{ #1 \ \middle|\ \sloppy #2 \right\}}
\newcommand{\ZZ}{\mathbb{Z}}
\newcommand{\NN}{\mathbb{N}}
\newcommand{\floor}[1]{\left\lfloor #1 \right\rfloor}
\newcommand{\ceil}[1]{\left\lceil #1 \right\rceil}
\newcommand{\ssquare}{\ \square\ }

\newcommand{\ie}{{\it i.e.\ }}

\newcommand{\resp}{resp.\ }

\newcommand{\G}{(G,\sigma)}

\newcommand{\sto}{\longrightarrow_s}

\newcommand{\vlonly}[1]{}

\newcommand{\aftercaldam}[1]{#1}

\title{On Cartesian products of signed graphs}
\author[labri]{Dimitri Lajou}
\address[labri]{Univ. Bordeaux, Bordeaux INP, CNRS, LaBRI, UMR5800, F-33400 Talence, France.\\email: dimitri.lajou@labri.fr}

\begin{document}

\begin{abstract}
   In this paper, we study the Cartesian product of signed graphs as defined by Germina, Hameed and Zaslavsky (2011). Here we focus on its algebraic properties and look at the chromatic number of some Cartesian products. One of our main results is the unicity of the prime factor decomposition of signed graphs. This leads us to present an algorithm to compute this decomposition in linear time based on a decomposition algorithm for oriented graphs by Imrich and Peterin (2018). 
    We also study the chromatic number of a signed graph, that is the minimum order of a signed graph to which the input signed graph admits a homomorphism, of graphs with underlying graph of the form $P_n \ssquare P_m$, of Cartesian products of signed paths, of Cartesian products of signed complete graphs and of Cartesian products of signed cycles.
\end{abstract}

\maketitle


\section{Introduction}
\label{sec:cartesian:intro}
%


Signed graphs were introduced by Harary in \cite{Harary1953}. 
In 2005, Guenin introduced the notion of homomorphism of signed graphs, which was later studied by Naserasr, Rollov\'{a} and Sopena \cite{Naserasr2014}. This gave rise to a notion of chromatic number $\chi_s\G$ of a signed graph $\G$ defined as the smallest order of a signed graph $(H,\pi)$ to which $(G,\sigma)$ admits a homomorphism.  

In this paper, we are interested in the study of Cartesian products of signed graphs, defined by Germina, Hameed and Zaslavsky in \cite{Germina2011}. They mainly study the spectral properties of the Cartesian product. In this paper, we present algebraic properties of the Cartesian product and study the chromatic number of some Cartesian products of signed graphs.

 The Cartesian product of two ordinary graphs $G$ and $H$, noted $G\ssquare H$, has been extensively studied. In 1957, Sabidussi \cite{Sabidussi1957} showed that $\chi(G \ssquare H) = \max(\chi(G),\chi(H))$ where $\chi(G)$ is the chromatic number of the graph $G$. Another notable article on the subject by Sabidussi \cite{Sabidussi1959} shows that every connected graph $G$ admits a unique prime decomposition, \ie{} there is a unique way to write a graph $G$ as a product of some graphs up to isomorphism of the factors. This result was also independently discovered by Vizing in \cite{Vizing1963}. Another algebraic property, the cancellation property, which states that if $A \ssquare B = A \ssquare C$, then $B = C$, was proved by Imrich and Klav\v{z}ar \cite{Imrich2007b} using a technique of Fern\'{a}ndez, Leighton and L\'{o}pez-Presa  \cite{Fernandez2007}.
On the complexity side, the main question associated with the Cartesian product is to decompose a graph with the best possible complexity. 
The complexity of this problem has been improved successively in \cite{Feigenbaum1985,Winkler1987,Feder1992,Aurenhammer1992} to finally reach an optimal complexity of $O(m)$ in \cite{Imrich2007a} where $m$ is the number of edges of the graph.

Our study of the Cartesian product of signed graphs is divided in several sections. 
First in section~\ref{sec:cartesian:def}, we present general definitions of graph theory and set our notation. 
In section~\ref{sec:cartesian:prelim}, we present some useful results on signed graphs and on the Cartesian product of ordinary graphs. 
In section~\ref{sec:cartesian:cartesian}, we present the definition of the Cartesian product of signed graphs and give some first properties and easy consequences of the definition. We also prove the prime decomposition theorem for signed graphs and give an algorithm to decompose a Cartesian product of signed graphs into its  factors.
%
%
We study the chromatic number of Cartesian products of signed complete graphs in section \ref{sec:cartesian:complete} and products of cycles in section \ref{sec:cartesian:cycles}.
Finally we present some concluding remarks in section~\ref{sec:cartesian:conclu}.


\section{Definitions and notation}
\label{sec:cartesian:def}
All graphs we consider are undirected, simple and loopless.
For classical graph definitions, we refer the reader to the book \emph{Graph Theory} by Bondy and Murty \cite{Bondy2008}.
%
%

Two vertices $u$ and $v$ of a graph $G$ are said to be {\em adjacent} when $uv$ is an edge of $G$. An edge $uv$ is {\em incident} with a vertex $w$ if and only if $w$ is one of $u$ or $v$.
The {\em neighborhood} $N_G(u)$ of a vertex $u$ in the graph $G$ is the set of vertices adjacent to $u$ in $G$. When the context is clear, we note $N(u)$ for the neighborhood of $u$ in $G$.
%
%
The {\em order} of $G$ is $\abs{V(G)}$ and its {\em size} is $\abs{E(G)}$ where $\abs{X}$ is the cardinal of a set $X$. 
A {\em proper $k$-vertex-coloring} of a graph $G$ is a function from $V(G)$ to the set of colors $\setk{k} = \set{1,\dots,k}$, such that no two adjacent vertices receive the same color.
The {\em chromatic number} $\chi(G)$ of a graph $G$ is the smallest $k$ such that $G$ admits a proper $k$-vertex-coloring.

A {\em homomorphism} of $G$ to $H$ is a function $\varphi$ from $V(G)$ to $V(H)$ such that for all $x,y \in V(G)$, $xy \in E(G)$ implies $\varphi(x)\varphi(y) \in E(H)$. When there is a homomorphism of $G$ to $H$, we note $G \to H$. Note that the chromatic number of $G$, $\chi(G)$, can also be defined as the smallest order of a graph $H$ such that $G \to H$.
An {\em isomorphism} of $G$ to $H$ is a bijection $\varphi$ from $V(G)$ to $V(H)$ such that for all $x,y \in V(G)$, $xy \in E(G)$ if and only if $\varphi(x)\varphi(y) \in E(H)$. In this case, we note $G = H$. 
%
%


A {\em walk} in a graph $G$ is a sequence $s_0,\dots,s_n$ of vertices of $G$ such that $s_i s_{i+1} \in E(G)$. Its {\em starting vertex} is $s_0$ and its {\em end vertex} is $s_n$. A {\em closed walk} is a walk where $s_0 = s_n$. If all elements of a walk are pairwise distinct, then the walk is a {\em path}. A closed walk where all elements are pairwise distinct, except $s_0$ and $s_n$, is a {\em cycle}. 
The length (number of edges, counted with multiplicity) of a walk $W$ = $s_0,\dots,s_n$ is $n$, and its order (number of vertices, counted with multiplicity) is $n$ if $W$ is a closed walk, or $n+1$ otherwise.

A graph is {\em connected} if for all pairs of vertices $u,v \in V(G)$, there is a path between $u$ and $v$. 
If $X \subseteq V(G)$, then the graph $G[X]$ is the subgraph of $G$ induced by $X$. We say that $G[X]$ is an {\em induced subgraph} of $G$.
%
The {\em complete graph} $K_p$ is the graph of order $p$ such that for all pair of distinct vertices of $G$, $u$ and $v$, $uv$ is an edge of $K_p$.
%

\vlonly{
\subsection{quotients}
\todo{ plein d'abus de notation plus bas: à corriger + déplacer dans section cycles}
For a group $(H,+,0)$, noted simply $H$, and a subgroup $Q$ of $H$, the {\em quotient} $\faktor{H}{Q}$ is the group $(\setcond{ \overline{x} }{ x \in H},+,\overline{0})$ where $\overline{x} = \setcond{y\in H}{y = x +q, q\in Q}$ and where the $+$ operation verify $\overline{x} + \overline{y} = \overline{x+y}$.
If $G$ is a graph with vertex set a group $H$ and $Q$ is a subgroup of $H$, then the {\em quotient graph} $\faktor{G}{Q}$ over the vertices $\faktor{H}{Q}$ is defined by identifying the vertices in the same equivalence class. Similarly if $W$ is a walk $s_0, \dots, s_n$ on $G$, then the {\em quotient walk} $W'$ on  $\faktor{G}{Q}$ is the sequence $\overline{s_0}, \dots, \overline{s_n}$.
}

\medskip

\aftercaldam{
A \textit{signed graph} $(G, \sigma)$ is a graph $G$ along with a function
$\sigma: E(G) \rightarrow \{+1,-1\}$ called the \textit{signature} of $(G,\sigma)$, where $\sigma(e)$ is the \textit{sign} of the edge $e \in E(G)$.
 The edges in $\sigma^{-1}(+1)$ are the \textit{positive} edges  and the edges in $\sigma^{-1}(-1)$ are the \textit{negative} edges  of $(G, \sigma)$. }
We often write a signed graph $\G$ as $(G,\Sigma)$ where $\Sigma$ is the set of negative edges, that is  $\Sigma = \sigma^{-1}(-1)$. These two ways to represent a signed graph are equivalent and will be used interchangeably. 
%
We note $K_p^+$ (\resp $K_p^-$) for the complete signed graph $(K_p,\varnothing)$ (\resp $(K_p, E(K_p))$) of order $p$ with only positive (\resp negative) edges.

\medskip

Let $(G,\sigma)$ be a signed graph and $v$ be a vertex of $G$. {\em To switch} $v$ is to create the signed graph $(G,\sigma')$ where $\sigma'(e) = - \sigma(e)$ when $e$ is incident to $v$ and $\sigma'(e) = \sigma(e)$ otherwise.  {To switch} a set $X$ of  vertices of $\G$ is to create the signed graph $(G,\sigma')$ where $\sigma'$ is obtained by switching every vertex of $X$, in any order.
This led Zaslavsky in \cite{Zaslavsky1982} to define the notion of equivalent signed graphs.
Two signed graphs $(G,\sigma_1)$ and $(G,\sigma_2)$ on the same underlying graph are {\em equivalent} if there exists a set $X \subseteq V(G)$ such that $(G,\sigma_2)$ is obtained from $(G,\sigma_1)$ by switching $X$. In this case we note $(G,\sigma_1) \equiv (G,\sigma_2)$. 
We also say that the two signatures $\sigma_1$ and $\sigma_2$ (\resp $\Sigma_1$ and $\Sigma_2$) are equivalent and we note $\sigma_1 \equiv \sigma_2$ (\resp $\Sigma_1 \equiv \Sigma_2$). 

\medskip

Suppose that $(G,\sigma)$ is a signed graph and $W$ is a walk $s_0,\dots,s_n$ in $G$. We say that $W$ is a {\em balanced walk} if $\sigma(W)= \sigma(s_0s_1)\sigma(s_1s_2)\dots\sigma(s_is_{i+1})\dots\sigma(s_{n-1}s_n) = 1$ and an {\em unbalanced walk} otherwise. Similarly, this notion can be extended to closed walks, paths and cycles. We note an unbalanced path (\resp balanced path) of order $k$  by $UP_k$ (\resp $BP_k$) and an unbalanced cycle (\resp balanced cycle) of order $k$ by $UC_k$ (\resp $BC_k$). 
A signed graph where all closed walks are balanced is said to be {\em balanced} while a signed graph where all closed walks are unbalanced is said to be {\em antibalanced}. Generally, for the same ordinary graph $G$, there are several signatures $\sigma$ for which $(G,\sigma)$ is balanced. They are precisely the signatures $\sigma_X$ which can be obtained from $(G,\varnothing)$ by switching $X$, where $X \subseteq V(G)$.  In particular it is the case for all signatures of a forest.
These notions of balanced and antibalanced graphs where introduced by Harary in \cite{Harary1953}. 

 One can check that the switch operation does not modify the set of balanced closed walks as switching at a vertex of a closed walk does not change the sign of this walk. 
Hence, signed graphs $(G,\sigma_1)$ and $(G,\sigma_2)$ on the same underlying graph are equivalent if and only if they have the same set of balanced closed walks \cite{Zaslavsky1982}. Note that this is equivalent to having the same set of balanced cycles, or the same set of unbalanced closed walks (resp. cycles). This means that we can work with the balance of closed walks or with switches depending on which notion is the easiest to use when treating equivalence of signed graphs.
%

%
%

\medskip
A {\em homomorphism} of a signed graph $(G,\sigma)$ to a signed graph $(H,\pi)$
is a homomorphism $\varphi$ of $G$ to $H$
which maps balanced (\resp unbalanced) closed walks of $(G,\sigma)$ to balanced (\resp unbalanced) closed walks of $(H,\pi)$.
Alternatively, a homomorphism of $(G,\sigma)$ to $(H,\pi)$ is a homomorphism of $G$ to $H$ such that there exists a signature $\sigma'$ of $G$ with $\sigma' \equiv \sigma$, such that if $e$ is an edge of $G$, then $\pi(\varphi(e)) = \sigma'(e)$.
When there is a homomorphism of $\G$ to $(H,\pi)$, we note $\G \sto (H,\pi)$ and say that $\G$ {\em maps to} $(H,\pi)$. Here $(H,\pi)$ is the {\em target graph} of the homomorphism.
%
When constructing a homomorphism, we can always fix a given signature of the target graph \cite{Naserasr2014}.


The {\em chromatic number} $\chi_s\G$ of a signed graph $\G$ is the smallest $k$ for which $\G$ admits a homomorphism to a signed graph $(H,\pi)$ of order $k$. 
Alternatively, a signed graph $(G,\sigma)$ admits a {\em $k$-(vertex)-coloring} if there exists $\sigma' \equiv \sigma$ such that $(G,\sigma')$ admits a proper vertex coloring $\theta:V(G) \to \setk{k}$ verifying that for every $i,j \in \setk{k}$, all edges $uv$ with $\theta(u) = i$ and $\theta(v) = j$ have the same sign in $(G,\sigma')$.
Here $\chi_s(G,\sigma)$ is the smallest $k$ such that $(G,\sigma)$ admits a $k$-vertex-coloring.
%
%
The two definitions are equivalent, as with any coloring of a signed graph, we can associate a homomorphism of signed graphs which 
identifies the vertices with the same color. 
The homomorphism is well defined as long as the target graph is simple, which is the case here by definition of a $k$-vertex-coloring.

\medskip

\medskip

The {\em Cartesian product} of two ordinary graphs $G$ and $H$ is the graph $G \ssquare H$ whose vertex set is $V(G) \times V(H)$ and where $(x,y)$ and $(x',y')$ are adjacent if and only if $x=x'$ and $y$ is adjacent to $y'$ in $H$, or $y = y'$ and $x$ is adjacent to $x'$ in $G$. 

A graph $G$ is {\em prime} if there are no graphs $A$ and $B$ on at least two vertices for which $G = A \ssquare B$. 
A decomposition $D$ of a graph $G$ is a multiset $\set{G_1,\dots,G_k}$, $k\geq 1$,  such that the $G_i$'s are graphs containing at least one edge and $G = G_1 \ssquare \cdots\ssquare G_k$.
A decomposition is {\em prime} if all the $G_i$'s are prime.
The $G_i$'s are called {\em factors} of $G$.
%
A decomposition $D'$ is \emph{finer} than a decomposition $D = \set{G_1,\dots,G_k}$, if for all $i \in \setk{k}$, there is a decomposition $D_i' = \set{G_{i,1}', \dots, G_{i,p_i}'}$ of $G_i$  such that 
 $D' = \set{G_{1,1}', \dots, G_{1,p_1}', G_{2,1}', \dots, G_{k,p_k}'}$. 
\aftercaldam{ Note that by definition, every decomposition is finer than itself.}

Suppose that $G$ is a graph and $D = \set{G_1,\dots,G_k} $ is a decomposition of $G$ such that $G = G_1 \ssquare \dots \ssquare G_k$. 
A {\em coordinate system} for $G$ under the decomposition $D$ is a bijection $\theta: V(G) \rightarrow \prod_{i = 1}^k V(G_i)$ verifying that for each vertex $v$ of $G$, the set of vertices which differ from $v$ by the $i$th coordinate induces a graph, noted $G_i^v$ and called a {\em $G_i$-layer}, which is isomorphic to $G_i$ by the projection on the $i$th coordinate.  An edge $uv$ of $G$ is a copy of an edge $ab$ of $G_i$ if $\theta(u)$ and $\theta(v)$ differ only in their $i$th coordinate with $u_i =a $ and $v_i=b$. For a vertex $u$ of $G$ and a $G_i$-layer $G_i^v$, the projection of $u$ on the $G_i$-layer $G_i^v$ is the vertex $w$ of $V(G_i^v)$ which is the closest to $u$. 

Suppose $D = \set{G_1,\dots,G_k}$ is a decomposition of an ordinary graph $G$.     We say that two $G_i$-layers $X_1$ and $X_2$ are {\em adjacent} by $G_j$ if and only if there exists an edge $ab$ of a  $G_j$-layer such that $a \in X_1$ and $b \in X_2$. In other words, the subgraph induced by the vertices of $X_1$ and $X_2$ is isomorphic to $G_i \ssquare K_2$ where  $K_2$ corresponds to the edge $ab$.

Let $A$ and $B$ be two ordinary graphs. The  \emph{greatest common divisor} of $A$ and $B$ is the graph $X$ such that, for every three graphs $W$, $Y$, and $Z$ with $A = W \ssquare Y$ and $B = W \ssquare Z$, $X$ is a factor of $W$.

\section{Preliminary results}
\label{sec:cartesian:prelim}
The goal of this section is to present useful results on signed graphs and on the Cartesian product of ordinary graphs.


In \cite{Zaslavsky1982}, Zaslavsky gave a way to determine if two signed graphs are equivalent in linear time. 
In particular, all signed forests with the same underlying graph are equivalent. 
This theorem comes from the following observation.

\begin{obs}[Zaslavsky \cite{Zaslavsky1982}]\label{lemma:balancing-and-homorphism}
If $C$ is a cycle of a graph $G$, then switching any number of vertices of $G$ does not change the parity of the number of negative edges of $C$.
\end{obs}

%
%
%

This implies that we can separate the set of all cycles into four families $BC_{even}$, $BC_{odd}$, $UC_{even}$ and $UC_{odd}$,  depending on the parity of the number of negative  edges (even for $BC_{even}$ and $BC_{odd}$ and odd for $UC_{even}$ and $UC_{odd}$) and the parity of the length of the cycle (even for $BC_{even}$ and $UC_{even}$ and odd for $BC_{odd}$ and $UC_{odd}$). 

\aftercaldam{
\begin{theorem}
\label{thm:cyclespasproduit}
Let $(C,\sigma)$ be a signed cycle. We then have:
\begin{enumerate}
    \item $\chi_s(C,\sigma) = 2$ if $(C,\sigma) \in BC_{even}$,
    \item $\chi_s(C,\sigma) = 3$ if $(C,\sigma) \in BC_{odd} \cup UC_{odd}$,
    \item $\chi_s(C,\sigma) = 4$ if $(C,\sigma) \in UC_{even}$.
\end{enumerate}
\end{theorem}}

\begin{proof}
By \cite{Duffy2020}, we already have the upper bounds.
A homomorphism of signed graphs is also a homomorphism of graphs thus $\chi(C) \leq \chi_s(C,\sigma)$. This proves the lowers bounds for the first two cases. 
%
Let $(C,\sigma) \equiv UC_{2q}$ and suppose $\chi_s(C,\sigma) \leq 3$. Then $(C,\sigma) \sto (K_3,\pi)$. In each case, $(K_3,\pi)$ can be switched either to be  all positive or to be all negative. This means that $(C,\sigma)$ can be switched either to be  all positive or to be all negative, which is not the case as $UC_{2q}$ has an odd number of negatives edges and an odd number of positive edges, a contradiction. We get the desired lower bounds in each case.
\end{proof}

One of the first results on the chromatic number of Cartesian products of ordinary graphs is due to Sabidussi:

\begin{theorem}[Sabidussi \cite{Sabidussi1957}]
\label{thm:ordinari-chi-prod}
For every two graphs $G$ and $H$, $\chi(G \ssquare H) = \max(\chi(G), \chi(H))$.
\end{theorem}

Following this paper, Sabidussi proved one of the most important results on the Cartesian product: the unicity of the prime decomposition of connected graphs. This result was  independently proved by Vizing.

\begin{theorem}[Sabidussi \cite{Sabidussi1959} and Vizing \cite{Vizing1963}]
\label{thm:ordinari-prime-decompostion}
Every connected ordinary graph $G$ admits a unique prime decomposition up to the order and isomorphisms of the factors.
\end{theorem}

Using some arguments of \cite{Fernandez2007} and the previous theorem, Imrich and Klav\v{z}ar proved the following theorem.

\begin{theorem}[Imrich and Klav\v{z}ar \cite{ImrichBook,Imrich2007b}]
\label{thm:cartesian:ordinari-semi-ring-cancellation}
If $A$, $B$ and $C$ are three ordinary graphs such that $A\ssquare B = A \ssquare C$, then $B = C$.
\end{theorem}

The unicity of the prime decomposition raises the question of the complexity of finding such a decomposition. The complexity of decomposition algorithms has been extensively studied. The first algorithm, by Feigenbaum {\it et al.} \cite{Feigenbaum1985}, had a complexity of $O(n^{4.5})$ where $n$ is the order of the graph (its size is denoted by $m$). In \cite{Winkler1987}, Winkler proposed a different algorithm improving the complexity to $O(n^4)$. Then Feder \cite{Feder1992} gave an algorithm in $O(mn)$ time and $O(m)$ space. The same year, Aurenhammer {\it et al.} \cite{Aurenhammer1992} gave an algorithm in $O(m \log n)$ time and $O(m)$ space. The latest result is an optimal algorithm.

\begin{theorem}[Imrich and Peterin \cite{Imrich2007a}]
The prime factorization of connected ordinary graphs can be found in $O(m)$ time and space. Additionally a coordinate system can be computed in $O(m)$ time and space.
\end{theorem}

\section{Cartesian products of signed graphs}
\label{sec:cartesian:cartesian}

\subsection{Definition}

We recall the definition of the Cartesian product of signed graphs due to Germina, Hameed K. and Zaslavsky:

\begin{definition}[\cite{Germina2011}]
Let $(G,\sigma)$ and $(H,\pi)$ be two signed graphs. The {\em Cartesian product} of $(G,\sigma)$ and $(H,\pi)$, denoted by $(G,\sigma) \ssquare (H,\pi)$, is the signed graph defined as follows:
\begin{itemize}
    \item $V((G,\sigma) \ssquare (H,\pi)) = V(G) \times V(H)$,
    \item the positive (\resp negative) edges are the pairs $\set{(u_1,v_1),(u_2,v_2)}$ such that:
    \begin{itemize}
        \item $u_1 = u_2$ and $v_1v_2$ is a positive (\resp negative) edge of $(H,\pi)$, or
        \item $v_1 = v_2$ and $u_1u_2$ is a positive (\resp negative) edge of $(G,\sigma)$. 
    \end{itemize} 
\end{itemize}
\end{definition}

Note that the underlying graph of $(G,\sigma)\ssquare (H,\pi)$ is the ordinary graph $G\ssquare H$.
From this definition, we can derive that the Cartesian product is associative and commutative. 
%
%


The following result shows that Cartesian products are compatible with  homomorphisms of signed graphs and in particular with the switching operation.

\begin{theorem}
\label{thm:cartesianHomCompatibility}
If $(G,\sigma)$, $(G',\sigma')$, $(H,\pi)$, $(H',\pi')$ are four signed graphs such that $(G,\sigma) \sto (G',\sigma')$ and $(H,\pi) \sto (H',\pi')$, then: $$(G,\sigma) \ssquare (H,\pi) \sto (G',\sigma') \ssquare (H',\pi').$$
\end{theorem}

\begin{proof}
By commutativity of the Cartesian product and composition of homomorphisms, it suffices to  show that $(G,\sigma) \ssquare (H,\pi) \sto (G',\sigma') \ssquare (H,\pi)$. Since $(G,\sigma) \sto (G',\sigma')$, there exists a set $S$ of vertices and a homomorphism $\varphi$ of $G$ to $G'$ such that if $(G,\sigma_S)$ is the signed graph obtained from $\G$ by switching the vertices of $S$, then $\sigma'(\varphi(e)) = \sigma_S(e)$ for every edge $e$ of $G$.
We note $P = (G,\sigma) \ssquare (H,\pi)$ and $X = \setcond{(g,h) \in V(G\ssquare H)}{ g \in S}$. Let $P'$ be the signed graph obtained from $P$ by switching the vertices in $X$.

If $(g,h)(g,h')$ is an edge of $P$, then in $P'$ this edge was  either switched twice if $g \in S$ or not switched if $g \notin S$. In both cases its sign did not change. If $(g,h)(g',h)$ is an edge of $P$, then in $P'$ this edge was   switched twice if $g,g' \in S$, switched once if $g \in S$, $g' \notin S$ or $g \notin S$, $g' \in S$, and not switched if $g,g' \notin S$. In each case its new sign is $\sigma_S(gg')$. Thus $P' = (G,\sigma_S) \ssquare (H,\pi)$. Now define $\varphi_P(g,h) = (\varphi(g),h)$. It is a homomorphism of $G \ssquare H$ to $G' \ssquare H$ by definition. By construction, the target graph of $\varphi_P$ is $(G',\sigma') \ssquare (H,\pi)$ as the edges of $H$ do not change and the target graph of $\varphi$ is $(G',\sigma')$.
\end{proof}

As mentioned before, we can derive the following corollary from Theorem~\ref{thm:cartesianHomCompatibility}.

\begin{corollary}
If $(G,\sigma)$, $(G,\sigma')$, $(H,\pi)$, $(H,\pi')$ are four signed graphs such that $\sigma \equiv \sigma'$ and $\pi \equiv \pi'$, then: $$(G,\sigma) \ssquare (H,\pi) \equiv (G,\sigma') \ssquare (H,\pi').$$
\end{corollary}

From Theorem~\ref{thm:cartesianHomCompatibility}, and the fact that $(F,\sigma) \sto K_2^+$ for every signed forest $(F,\sigma)$, we also get the following corollary:

\begin{corollary}
If $\G$ is a signed graph and $(F,\pi)$ is a signed forest with at least one edge, then:
$$ \chi_s(\G\ssquare (F,\pi)) = \chi_s(\G\ssquare K_2^+). $$
In particular, for $n,m \geq 2$, $\chi_s((P_n,\sigma_1) \ssquare (P_m,\sigma_2)) = 2$.
\end{corollary}


\subsection{Signed grids}

Note that there is a difference between considering the chromatic number of the Cartesian product of two signed graphs and the chromatic number of a signed graph whose underlying graph is a Cartesian product. 
For example, $C_4 = K_2 \ssquare K_2$ but $4 = \chi_s(UC_4) \neq \chi_s(BC_4) = 2$.
Another example comes from grid graphs: $\chi_s((P_n,\sigma_1) \ssquare (P_m,\sigma_2)) = 2$, for any $n,m \in \NN$, but the following theorem shows that not all signed grids have chromatic number $2$.

\begin{theorem}
\label{th:grid}
If $n$ and $m$ are two integers with $1 \leq n \leq m$ and $(G,\sigma)$ is a signed grid with $G = P_n\ssquare P_m$, then $\chi_s(G) \leq 6$. If $n \leq 4$, then $\chi_s(G) \leq 5$. Moreover there exist signed grids with chromatic number $5$.
\end{theorem}

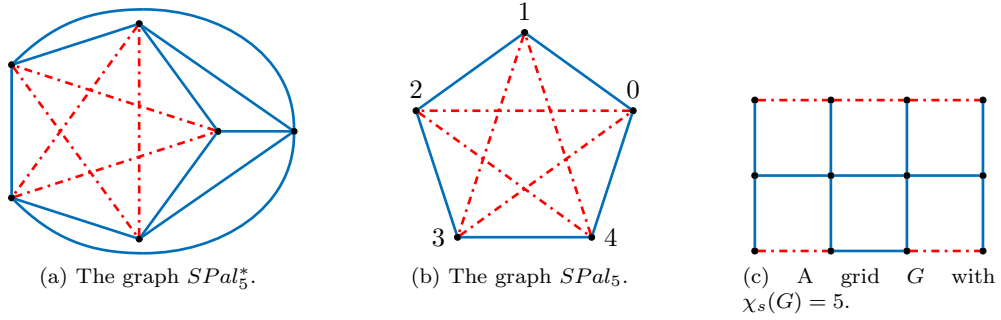
\begin{figure}
\centering
\subfloat[The graph $SPal^*_5$. \label{fig:spal5*}]{
\begin{tikzpicture}
\tikzstylemacro{}

\node[n] (1) at (0:1.5) {};
\node[n] (2) at (72:1.5) {};
\node[n] (3) at (144:1.5) {};
\node[n] (4) at (216:1.5) {};
\node[n] (5) at (288:1.5) {};

\node[n] (0) at (0:2.5) {};

\draw[g] (1) -- (2) -- (3) -- (4) -- (5) -- (1);
\draw[r] (1) -- (3) -- (5) -- (2) -- (4) -- (1);

\coordinate (i) at (72:1.7) {};
\coordinate (j) at (288:1.7) {};

\draw (0) 	edge[g] (1) 
			edge[g] (2) 
			edge[g, in =0, out = 90] (i) 
			edge[g, in =0, out=270] (j) 
			edge[g] (5);
			
\draw (3) edge[g, in =180, out=45] (i);
\draw (4) edge[g, in=180, out = 315] (j);
\end{tikzpicture}
}\hspace{1cm}
\subfloat[The graph $SPal_5$. \label{fig:spal5}]{
\begin{tikzpicture}
\tikzstylemacro{}

\node[n, label=$0$] (1) at (18:1.5) {};
\node[n, label=$1$] (2) at (72 +18:1.5) {};
\node[n, label=$2$] (3) at (144+18:1.5) {};
\node[n, label=left:$3$] (4) at (216+18:1.5) {};
\node[n, label=right:$4$] (5) at (288+18:1.5) {};

\draw[g] (1) -- (2) -- (3) -- (4) -- (5) -- (1);
\draw[r] (1) -- (3) -- (5) -- (2) -- (4) -- (1);

\end{tikzpicture}
}\hspace{1cm}
\subfloat[A grid $G$ with $\chi_s(G) = 5$.  \label{fig:grid-counterexample}]{
\begin{tikzpicture}
\tikzstylemacro{}

\foreach \i in {1,2,3,4}{
	\foreach \j in {1,2,3}{
		\node[n] (\i\j) at (\i,\j) {};
	}
	\draw[g] (\i1) -- (\i2) -- (\i3);
}

\draw[g] (12) -- (22) -- (32) -- (42);
\draw[g] (21) -- (31);
\draw[r] (11) -- (21);
\draw[r] (31) -- (41);
\draw[r] (13) -- (23) -- (33) -- (43);

\end{tikzpicture}
}
\caption{The signed graphs used in the proof of Theorem~\ref{th:grid}.}
\label{fig:grid}
\end{figure}
On our figures, we use dashed red edges to represent negative edges and solid blue edges for positive edges. 

\begin{proof}
We will prove a more precise statement: every signed grid $(G,\sigma)$ verifies $(G,\sigma) \sto SPal^*_5$ where $SPal^*_5$ is the graph of Figure~\ref{fig:spal5*}. This graph has the following (easy to check) property: 
\begin{itemize}
    \item[($\mathcal{P}$)] for every three vertices $x$,$y$,$z$ of $SPal^*_5$, and every sign $\epsilon \in \set{+1,-1}$, if $x \neq z$ or $\epsilon = +1$ then  there exists $u$ and $v$ in $SPal^*_5$, $u \neq v$, such that the cycles $xyzu$ and $xyzv$ have sign $\epsilon$.
\end{itemize}

 To map $(G,\sigma)$ to $SPal^*_5$, we will construct the homomorphism $\varphi$ by induction. The vertex  of $G$ in line $i \in \set{1,\dots,n}$ and column $j \in \set{1,\dots,m}$ will be called $x_{i,j}$. Let $H_{i,j}$ be the subgraph of $G$ induced by the vertices $x_{k,\ell}$ where $k < i$, or $k = i$ and $\ell \leq j$.
We prove that for all $i,j$, $0\leq i \leq n$ and $0\leq j \leq m$,  $H_{i,j} \sto Spal^*_5$.
 It is easy to see that $H_{0,m} \sto SPal^*_5$.
 
If $H_{i,m} \sto Spal^*_5$ and $i<n$, then $x_{i+1,0}$ has only one neighbor in $H_{i+1,0}$ and we can extend the previous homomorphism to $H_{i+1,0}$.
 
 Suppose that $\varphi$ is a homomorphism of $H_{i,j}$ to $Spal^*_5$ and $j<m$.
 Let  $C = x_{i,j+1}x_{i,j}x_{i-1,j}x_{i-1,j+1}$. If $C = BC_4$ or if $C = UC_4$ and $\varphi(x_{i,j}) \neq \varphi(x_{i-1,j+1})$, then we have two choices for $x_{i,j+1}$ by $\mathcal P$ (we might need to switch $x_{i,j}$) and we can extend the homomorphism to $H_{i,j+1}$.  
 If $C = UC_4$ and $\varphi(x_{i,j}) = \varphi(x_{i-1,j+1})$, then  these two vertices must be different. There were two possibilities for the choice of $\varphi(x_{i,j-1})$ in the previous step by  $\mathcal P$ thus if we take the other one, we are back to the previous case where $\varphi(x_{i,j-1}) \neq \varphi(x_{i-1,j})$.
 Thus we can extend $\varphi$ to $H_{i,j+1}$.
 
 Hence,  $H_{n,m} = (G,\sigma) \sto SPal^*_5$  which gives $\chi_s(G,\sigma) \leq 6$.
 
 \ 
 
 Suppose now that $n \leq 4$. We construct a homomorphism $\varphi : (G,\sigma) \sto SPal_5$, column by column, where $SPal_5$ is the graph of Figure~\ref{fig:spal5}. The first column is a path and thus, we can map it arbitrarily to $SPal_5$. For a column with vertices $x_{1,j}, x_{2,j}, x_{3,j}, x_{4,j}$, there are at least three  possibilities to map $x_{1,j}$ (the three colors different from $\varphi(x_{1,j-1})$ and $\varphi(x_{2,j-1})$). Up to symmetry, we can suppose that $x_{1,j-1}$ has color $0$ and $x_{1,j-1}$ has color $1$. These three possibilities ($2$, $3$, $4$) give at least three possibilities for $x_{2,j}$ where we need to remove one of them to account for the possibility of a $UC_4$ forcing $x_{2,j}$ and $x_{3,j-1}$ to have different images. Indeed, let $C = x_{1,j}x_{2,j}x_{2,j-1}x_{1,j-1}$, if  $C = BC_4$ then $x_{2,j}$ can have colors $1$, $2$ or $3$. If $C = UC_4$, then  $x_{2,j}$ can have colors $2$, $3$ or $4$.
 
 Again for $x_{3,j}$ there are at least two possibilities by the same kind of arguments. We need to remove one of them to account for the possibility of a $UC_4$ forcing $x_{3,j}$ and $x_{4,j-1}$ to have different colors. Finally there is at least one possibility for $x_{4,j}$. Thus we can extend our homomorphism to this column.
 This implies that $G \sto SPal_5$. 
 
 It is tedious but not difficult to check that the signed grid of Figure~\ref{fig:grid-counterexample} cannot be mapped to a signed graph of order $4$, thus its chromatic number if at least $5$. In fact it is exactly $5$. This concludes the proof.
\end{proof}

The arguments for the existence of a homomorphism to $SPal_5$ cannot be extended to bigger grids as we could end up in the case where $x_{4,j}$ has no possible image. We do not know if the upper bound for grids is $5$ or $6$. 


\begin{question}
What is the maximal value of $\chi_s(G,\sigma)$ when $(G,\sigma)$ is a signed grid? Is it $5$ or $6$?
\end{question}

\subsection{Prime factor decomposition}
 
Our goal now is to prove that each connected signed graph has a unique prime $s$-decomposition. Let us start with some definitions.


\begin{definition}
A signed graph $(G,\sigma)$ is said to be {\em $s$-prime} if and only if there do not exist two signed graphs $(A,\pi_A)$ and $(B,\pi_B)$ such that $(G,\sigma) \equiv (A,\pi_A)\ssquare (B,\pi_B)$. 
An {\em $s$-decomposition}  of a signed connected graph $(G,\sigma)$ is a multiset of signed graphs $D = \set{(G_1,\pi_1),\dots,(G_k,\pi_k)}$ such that:
\begin{enumerate}
\item the $(G_i,\pi_i)$'s are signed graphs containing at least one edge and
\item $(G,\pi) \equiv (G_1,\pi_1) \ssquare \cdots\ssquare (G_k,\pi_k)$.
\end{enumerate}
An $s$-decomposition $D$ is {\em prime} if all the $(G_i,\pi_i)$'s are $s$-prime.
The $(G_i,\pi_i)$'s are called {\em factors} of $D$. 
%

\end{definition}


Note that if $G = A \ssquare B$, then it is not always true that $(G,\sigma)$ is the Cartesian product of two signed graphs. For example, $UC_4$ is $s$-prime but $C_4$ is not a prime graph as $C_4 = K_2 \ssquare K_2$.
The following lemma tells us in which cases $(G,\sigma) \equiv (A,\pi_A) \ssquare (B,\pi_B)$, and will be a useful tool for decomposing signed graphs. 
\begin{lemma}
\label{lem:decomp}
If $(G,\sigma)$, $(A,\pi_A)$ and $(B,\pi_B)$ are three connected signed graphs with $G = A\ssquare B$, then  $(G,\sigma) \equiv (A,\pi_A) \ssquare (B,\pi_B)$ if and only if:
\begin{enumerate}
\item all $A$-layers are equivalent to $(A,\pi_A)$,
\item at least one $B$-layer is equivalent to $(B,\pi_B)$, and
\item for each edge $e$ of $A$ and each pair of distinct copies $e_1$,$e_2$ of $e$, if $e_1$ and $e_2$ belong to the same $4$-cycle, then this cycle is a $BC_4$.
\end{enumerate}
\end{lemma}

Note that, in the previous lemma, all $B$-layers are equivalent to $(B,\pi_B)$ but we only need to verify that for one of them to conclude. 

\begin{proof}

($\Rightarrow$) This follows from the definition of the Cartesian product.

($\Leftarrow$) We will do the following independent switches: switch all $A$-layers to have the same signature $\pi_A$.

 Now we claim that all $B$-layers have the same signature $\pi_B'$ equivalent to $\pi_B$. Indeed take one edge $xy$ of $B$ and two copies of this edge $x_1y_1$ and $x_2y_2$ in $G$. Take a shortest path $P$ from $x_1$ to $x_2$ in the $A^{x_1}$-layer. Now if $u_1$,$u_2$ are two consecutive vertices along $P$ and $v_1$ and $v_2$ are their projections on $A^{y_1}$, then $u_1u_2v_2v_1$ is a $BC_4$ by the third hypothesis as $u_1v_1$ and $u_2v_2$ are copies of the edge $xy$.
 
  As $u_1u_2$ and $v_1v_2$ have the same sign by the previous switches, it must be that $u_1v_1$ and $u_2v_2$ have the same sign. Thus all copies of an edge of $B$ have the same sign. 
  
  Hence, $(G,\sigma) \equiv (A,\pi_A)\ssquare (B,\pi_B') \equiv (A,\pi_A)\ssquare (B,\pi_B)$ by Theorem~\ref{thm:cartesianHomCompatibility}.
\end{proof}

One of our main results is the following Prime Decomposition Theorem.

\begin{theorem}[Prime Decomposition Theorem]
\label{thm:primefactortheorem}
If $(G,\sigma)$ is a connected signed graph and $D$ is the prime decomposition of $G$, then $(G,\sigma)$ admits a unique (up to isomorphism and order of the factors) prime $s$-decomposition $D_s$. Moreover, if we see $D_s$ as a decomposition of $G$, then $D$ is finer than $D_s$.
\end{theorem}

For proving this theorem, we need the following lemma.

\begin{lemma}
\label{lem:common-factor}
If $(G,\sigma)$ is a connected signed graph that admits two prime $s$-decompositions $D_1$ and $D_2$, then there are two signed graphs $(X,\pi_X)$ and $(Y,\pi_Y)$ such that $(G,\sigma) \equiv (X,\pi_X)\ssquare (Y,\pi_Y)$ with $D_1 = \set{(X,\pi_X)} \cup D_1'$ and $D_2 = \set{(X,\pi_X)} \cup D'_2$, where $D'_1$ and $D'_2$ are two s-decompositions of $(Y,\pi_Y)$.
\end{lemma}

\begin{proof}
Suppose there exists a signed graph $(G,\sigma)$ that admits two $s$-decompositions $D_1$ and $D_2$.
Fix an edge $e$ of $(G,\sigma)$ which belongs to some $Z$-layer $Z^e$ of the prime decomposition of $G$. The edge $e$ belongs to some $(A,\pi_A)$-layer in $D_1$ and to some $(B,\pi_B)$-layer in $D_2$. The graph $Z$ is a factor of $A$ and $B$ by unicity of the prime factor decomposition of $G$. Let $X$ be the greatest common divisor of $A$ and $B$. Since $e \in E(Z^e)$, $e$ is in some $X$-layer $X^e$. Now $G = X\ssquare Y$ for some graph $Y$. Let us show that $(G,\sigma) \equiv (X,\pi_X)\ssquare (Y,\pi_Y)$ for some signatures $\pi_X$ and $\pi_Y$ of $X$ and $Y$, respectively. We can suppose that $Y \neq K_1$ and $A \neq B$, as otherwise the result is immediate.

First we want to show that all $X$-layers have equivalent signatures. Take two adjacent $X$-layers. If they are in different $A$-layers, then they are equivalent since they represent the same part of $(A,\pi_A)$.
If they are in the same $A$-layer, then they are in different $B$-layers since $X$ is the greatest common divisor of $A$ and $B$. The same argument works in this case. Thus two adjacent $X$-layers are isomorphic  to the same signed graph $(X,\pi_X)$, and since there is only one connected component in $Y$, all $X$-layers have equivalent signatures. 

Let $\pi_Y$ be the signature of one $Y$-layer. Fix $e'$ an edge of $X$, and $X_1$, $X_2$ two 
 $X$-layers. Now consider the $4$-cycle (if it exists) containing the copies of this edge in each of the two layers. If $X_1$ and $X_2$ are in different $A$-layers, then this cycle is a $BC_4$ by Lemma~\ref{lem:decomp}, otherwise this cycle is a $BC_4$ as $X_1$ and $X_2$ are in different $B$-layers by the same argument.

By Lemma~\ref{lem:decomp}, we can conclude that $(G,\sigma) \equiv (X,\pi_X)\ssquare (Y,\pi_Y)$. 

Now suppose that $A = X \ssquare W$. Using Lemma~\ref{lem:decomp}, we can show that $(A,\pi_X) \equiv (X,\pi_X) \ssquare (W,\pi_W)$. Indeed all $X$-layers have equivalent signatures since $(G,\sigma) \equiv (X,\pi_X)\ssquare (Y,\pi_Y)$ and all $4$-cycles between two copies of an edge of $X$ are $BC_4$ by the same argument. As $(A,\pi_A)$ is $s$-prime, this implies $(X,\pi_X) \equiv (A,\pi_A)$.
 Thus $(X,\pi_X) \equiv (A,\pi_A) \equiv (B,\pi_B)$ and this proves the lemma.
\end{proof}

\begin{proof}[Proof of Theorem~\ref{thm:primefactortheorem}]
Any signed graph $\G$ has a prime $s$-decomposition by taking an $s$-decomposition that cannot be refined.
Every prime $s$-decomposition of $(G,\sigma)$  can be considered as a decomposition of $G$, and the prime decomposition of $G$ is finer than every such decomposition.
We still have to show that the prime $s$-decomposition of $(G,\sigma)$ is unique. Suppose, to the contrary, that $(G,\sigma)$ is a minimal counterexample to the unicity. Thus $(G,\sigma)$ has two prime $s$-decompositions $D_1$ and $D_2$ and, 
by Lemma~\ref{lem:common-factor}, $(G,\sigma) \equiv (X,\pi_X)\ssquare (Y,\pi_Y)$ with $D_1 = \set{(X,\pi_X)} \cup D_1'$ and $D_2 = \set{(X,\pi_X)} \cup D'_2$, where $D'_1$ and $D'_2$ are two s-decompositions of $(Y,\pi_Y)$. By minimality of $(G,\sigma)$,  $(Y,\pi_Y) $ has a unique prime s-decomposition, hence $D'_1 = D'_2$. Thus $D_1 = D_2$, a contradiction.
\end{proof}

Note that Theorem \ref{thm:primefactortheorem} implies the following result.

\begin{theorem}
\label{thm:cancellation-property}
If $(A,\pi_A)$, $(B,\pi_B)$ and $(C,\pi_C)$ are three signed graphs verifying $ (A,\pi_A) \ssquare (B,\pi_B) \equiv (A,\pi_A) \ssquare (C,\pi_C) $, then $(B,\pi_B) \equiv (C,\pi_C)$.
\end{theorem}

The proof of this result is exactly the same as the proof for ordinary graphs presented in \cite{Imrich2007b}. Indeed, we have all the necessary tools used in the proof. The first one is Theorem \ref{thm:primefactortheorem}, the other one is the semi-ring structure of signed graphs (quotiented by the equivalence relation) with the disjoint union and the Cartesian product which follows from the definition. See \cite{Imrich2007b} for more details on the proof.

\subsection{Recognising Cartesian products of signed graphs}
\medskip
In the last part of this section, we propose an algorithm to decompose connected signed graphs. Decomposing a graph can be interpreted in multiple ways: finding a decomposition, identifying which edge of $G$ belongs to which factor, or even better getting a coordinate system that is compatible with the decomposition.
In \cite{Imrich2007a}, Imrich and Peterin gave an $O(m)$ time and space ($m$ is the number of edges of $G$) algorithm for these three questions for ordinary graphs. 
More recently, in \cite{Imrich2018}, they gave another algorithm in $O(m)$ time and space to decompose directed graphs.

Our goal is to give a similar algorithm for signed graphs based on their algorithm for directed graphs.


\begin{theorem}\label{thm:algo-decomp}
Let $(G,\sigma)$ be a connected signed graph of order $n$ and size $m$. We can find in time $O(m)$ and space $O(m)$ the prime $s$-decomposition of $(G,\sigma)$ and a coordinate system for this decomposition. 
\end{theorem}

\begin{algorithm}[t]
 \SetKwData{Left}{left}\SetKwData{This}{this}\SetKwData{Up}{up}
 \SetKwFunction{Union}{Union}\SetKwFunction{FindCompress}{FindCompress}
 \SetKwInOut{Input}{Input}\SetKwInOut{Output}{Output}
 
 \Input{A signed graph $(G,\sigma)$}
 \Output{the prime factor $s$-decomposition of $(G,\sigma)$}
 Compute the prime factor decomposition $D$ of $G$\;
 Set the temporary decomposition of $(G,\sigma)$ to be $J = D$\;
 $Done \leftarrow \varnothing$\;
 $S \leftarrow \varnothing$\;
 $Treated \leftarrow \varnothing$\;
 \ForAll{vertices $x$ taken according to a BFS ordering}{
 	Add $x$ to $S$\;
 	\ForAll{edges $xy \notin Treated$\label{algo:line-8}}{
 		Determine the temporary color $i$ of $xy$ and $J_i$ the current factor to which it belongs in the current decomposition\;
 		Let $x'y'$ be the projection of $xy$ onto $J_i^v$\;
 		\uIf{$xy$ and $x'y'$ do not have the same sign and $y \notin S$\label{algo:line-11}}{
 			Switch the vertex $y$\;
 			Add $y$ to $S$\;
 		}
 		\uElseIf{$xy$ and $x'y'$ have the same sign and $y \notin S$\label{algo:line-14}}{
 			Add $y$ to $S$\;
 		}
 		\ElseIf{$xy$ and $x'y'$ do not have the same sign and $y \in S$\label{algo:line-16}}{
 			Merge the temporary colors of all up-edges of $y$ (and the temporary color of $xy$) and update the decomposition\;\label{algo:line-17}
 		}
 		Add $xy$ to $Treated$\;
 	}
 	Add $x$ to $Done$\;
 }
 \caption{A decomposition algorithm for signed graphs.}
 \label{algo}
\end{algorithm}

We take a coordinate system for a graph $G$ corresponding to its prime decomposition $D$ which can be computed in $O(m)$ time  \cite{Imrich2007a}. 
Let $v$ be the vertex of $G$ with coordinates all equal to zero. We order the vertices using a BFS traversal of the graph starting at $v$. If $xy$ is an edge, then it is a down-edge (resp. up-edge, resp. cross-edge) of $x$ when $d(v,x) < d(v,y)$ (resp. $d(v,x) > d(v,y)$, resp. $d(v,x) = d(v,y)$) where $d$ denotes the distance in $G$. 
We proceed as described in Algorithm~\ref{algo}. We color the edges of $G$ using the prime decomposition $D$ of $G$: we associate to each factor $X$ of $D$ a color, which is then assigned to every edge belonging to an $X$-layer of $G$. 
We maintain a temporary decomposition $J$ of $G$ for which  we merge some factors, by means of recoloring the edges, during the algorithm. 
Our goal, at the end of the algorithm, is that $J = P$ where  $P$ is the prime $s$-decomposition of $(G,\sigma)$. We note $p_i(e)$ the projection of an edge $e = xy \in J_i^x$ to the temporary $J_i$-layer $J_i^v$.

First note that in Algorithm~\ref{algo}, the set $Done$ is not used. Therefore, it can be omitted. Its only purpose is to ease the correctness analysis of the algorithm.
Let us make a few more remarks. The set $Done$ (resp. $Treated$) is used to record which vertex (resp. edge) has been processed by the algorithm. The set $S$ corresponds to the set of vertices for which we have decided whether they need to be switched or not.  If $x \in Done$ at some point of the algorithm then all its incident edges belong to the set $Treated$. Moreover, by construction of the BFS ordering, if $xy$ is a down-edge of $x$ in $J_i^x$ then, for every vertex $z$, the projection $x'y'$ of $xy$ on $J_i^z$ is a down-edge of $x'$.

\begin{claim}
\label{claim:merging}
After the merging in line~\ref{algo:line-17} of the algorithm, $v$, $y$ and $x$ belong to the same layer.
\end{claim}

\begin{proof}
We just need to prove that $y$ and $v$ belong to the same layer after merging. Note that a layer $J_i^a$ corresponds to all the vertices $b$ which differ from $a$ only by the $i$th coordinate (in the current decomposition). Note also that the coordinate vector of a neighbor of $y$ and the coordinate vector of $y$ differ by only one coordinate. For any non-zero coordinate of $y$, there is an up-edge $yz$ of $y$ to a neighbor $z$ of $y$ which differs only on this non-zero coordinate (as the ordering is a BFS ordering and by the Cartesian product structure), therefore all factors $J_\ell$ corresponding to non-zero coordinates of $y$ are merged. Hence, in this new coordinate system, $y$ has at most one non-zero coordinate and thus $y$ and $v$ are in the same layer.
\end{proof}

\begin{claim}\label{claim:treated}
Let $ab$ and $a'b'$ be two edges of the set $Treated$ at any moment of the algorithm. If $a'b' \in J_i^{a'}$ for some $i$ and $p_i(a'b') = p_i(ab)$ (\ie  they represent the same edge of $J_i$), then $ab$ and $a'b'$ have the same sign.
\end{claim}

\begin{proof}
By contradiction, suppose that $a'b'$ is the first edge such that, when added to $Treated$, there exists some edge $ab \in Treated$ such that $p_i(a'b') = p_i(ab)$ and $a'b'$ and $ab$ do not have the same sign. Let $a''b''$ be the edge $p_i(a'b')$. Note that no edge in $Treated$ can change sign once it is into the set as both its endpoints are in $S$. By definition of $a'b'$, $ab$ and $a''b''$ have the same sign since they both project to $a''b''$. Hence, it must be that $a'b'$ and $a''b''$ do not have the same sign.

Note that, $a'b'$ cannot be treated in the third if statement at line~\ref{algo:line-16}, as otherwise it would belong to some layer $J_i^v$ after merging by Claim \ref{claim:merging} and thus $a'b'$ would  project to itself.
Since $a'b'$ went through one of the first two if statements (lines~\ref{algo:line-11} and \ref{algo:line-14}), $a'b'$ and $a''b''$ have the same sign, a contradiction.
\end{proof}

\begin{proof}[Proof of Theorem~\ref{thm:algo-decomp}]
\mbox{}\\
\noindent\textbf{Correctness:}
First, let us show that $J$ is finer than $P$, the prime $s$-decomposition of $\G$, at each step of the algorithm. It is true at the beginning of the algorithm by Theorem~\ref{thm:primefactortheorem} as $J = D$. 
Suppose that $J$ is finer than $P$ at the beginning of step~\ref{algo:line-8}. In the if statement, if we enter the first two cases then we do not change $J$. Hence it is still finer than $P$ at the end of the loop.

Suppose $xy$ and $x'y'$ are not of the same sign and $y\in S$ (\ie we enter line~\ref{algo:line-17}). As $y \in S$ and $xy \notin Treated$, there is some neighbor $z$ of $y$ for which $z \in Done$. We consider two cases depending on whether $z \in J_i^x$ or $z \notin J_i^x$.

\medskip

Suppose first that $z \in J_i^x$. 

Take a shortest path $P_z$ in $J_i^x$ from $z$ to the projection $p_v$ of $v$ on $J_i^x$. All vertices of the path appear before $z$ in the BFS ordering, thus all the edges of the path belong to the set $Treated$. The same holds for a shortest path $P_x$ from  $p_v$ to $x$. In particular the  walk $W$ obtained by concatenating $yz$, $P_z$ and $P_x$ has all its edges in $Treated$. This implies that $W$ and $W'$, its projection on $J_i^v$, have the same sign by Claim~\ref{claim:treated}. Hence the closed walk $C$ obtained by concatenating $W$ with $xy$ and its projection ($W'$ concatenated with $x'y'$) have different signs and $J_i^v$ and $J_i^x$ do not have the same signature. 

Let $u$ be a neighbor of $y$ such that $uy$ is an up-edge of $y$ and $u \notin J_i^x$. Every edge $e'$ of the projection $C'$ of $C$ on $J_i^u$ is in $Treated$ as  $d(v,e') < d(v,e)$ where $e$ is the counterpart of $e'$ in $C$ (all vertices of $C$ have an up-edge to their projection on $J_i^u$). In particular $C$ and $C'$ do not have the same sign and $J_i^x$ and $J_i^u$ do not have the same signature. 

This implies that both layers are in the same factor of $P$. Indeed suppose that this is not the case. Then all cycles $abb'a'$, such that $ab \in J_i^x$ and $a'b'$ is its projection on $J_i^u$, must be $BC_4$. For all edges $ab$ of $W$, $ab$ and $a'b'$ have the same sign by Claim~\ref{claim:treated}, hence $aa'$ and $bb'$ also have the same sign (since the cycle is balanced). Now  let $x''y''$ be the projection of $xy$ on $J_i^u$.  By going around $W$ and by the previous observation, $xx''$ and $yy''$ have the same sign. Note that $xy$ and $x''y''$ do not have the same sign as $x''y'' \in Treated$ ($x''y''$ has the same sign as $x'y'$). This implies that $xyy''x''$ is a $UC_4$, a contradiction.

Hence we need to merge all temporary colors of all up-edges of $y$ (including color $i$). Thus after this step $J$ is still finer than $P$.

\medskip 
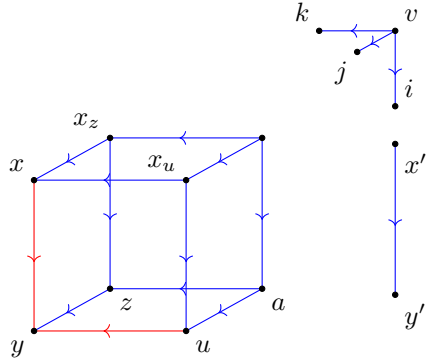
\begin{figure}
    \centering
    \begin{tikzpicture}[scale=1,z={(-.5,-.28)}]



    \draw (1,-1,1)   coordinate (A);
    \draw (-1,-1,1)  coordinate (B);
    \draw (-1,1,1)   coordinate (C);
    \draw (1,1,1)    coordinate (D);
    \draw (1,1,-1)   coordinate (E);
    \draw (-1,1,-1)  coordinate (F);
    \draw (-1,-1,-1) coordinate (G);
    \draw (1,-1,-1)  coordinate (H);
    
    \draw (2,2,-2.5)  coordinate (V);
    \draw (2,1,-2.5)  coordinate (I);
    \draw (2,2,-1.5)  coordinate (J);
    \draw (1,2,-2.5)  coordinate (K);
    
    \draw (2,0.5,-2.5)  coordinate (X');
    \draw (2,-1.5,-2.5)  coordinate (Y');
    
    \begin{scope}[decoration={
    markings,
    mark=at position 0.5 with {\arrow{<}}}]
    \draw[blue,postaction={decorate}]  (A)--(H);
    \draw[blue,postaction={decorate}]  (B)--(G);
    \draw[blue,postaction={decorate}]  (D)--(E);
    \draw[blue,postaction={decorate}]  (C)--(F);
    \draw[blue,postaction={decorate}]    (F)--(E);
    \draw[blue,postaction={decorate}] (G)--(F);
    \draw[blue,postaction={decorate}] (G)--(H);
    \draw[blue,postaction={decorate}] (H)--(E);
    \draw[blue,postaction={decorate}]  (A)--(D);
    \draw[blue,postaction={decorate}] (C)--(D);
    \draw[red,postaction={decorate}]   (B)--(A);
    \draw[red,postaction={decorate}] (B)--(C);
    \draw[blue,postaction={decorate}] (I)--(V);
    \draw[blue,postaction={decorate}] (J)--(V);
    \draw[blue,postaction={decorate}] (K)--(V);
    \draw[blue,postaction={decorate}] (Y')--(X');
    \end{scope}

    \foreach \l/\a/\t in {A/north west/u,B/north east/y,C/south east/x,D/south east/x_u,E/south west/ ,F/south east/x_z,G/north west/z,H/north west/a,V/south west/v,I/south west/i,J/north east/j,K/south east/k,X'/north west/x',Y'/north west/y'}
      \draw[fill=black] (\l) circle (1pt) node[anchor=\a] {$\t$};
  \end{tikzpicture}
  \caption{The second case of the correctness analysis. For simplicity, all edges which are in $Treated$ are positive. The orientation of the edges represents the BFS order. The neighbors of $v$ are labelled with the temporary color of their edge with $v$.}
  \label{fig:algo-case-2}
\end{figure}

Suppose now that $z \notin J_i^x$ (see Figure~\ref{fig:algo-case-2}). 

In this case, $z$ is the projection of $y$ on $J_i^z$. Let $x_z$ be the projection of $x$ on $J_i^z$. Since $yz$ is an up-egde of $y$, $xx_z$ is an up-edge of $x$ and $x_z \in Done$. Note that $xx_z$ and $yz$ have the same sign since both are in $Treated$. Also note that $x_zz$ and $x'y'$ have the same sign since $x_zz \in Treated$. Hence $xyzx_w$ is a $UC_4$. By the same arguments as before, these four vertices belong to the same signed factor of $(G,\sigma)$, hence we must merge $i$ and, say $j$, the temporary colors  of $xy$ and $yz$ respectively.

Let $u$ be a neighbor of $y$ such that $uy$ is an up-edge of $y$ of temporary color $k\notin \set{i,j}$. Let $x_u$ be the projection of $x$ on $J_i^u$. 
Note that $x_uu$ and $x'y'$ have the same sign as $d(x_u,v) < d(x,v)$  (\ie $x_u \in Done$). 
If $xx_u$ and $yu$ have the same sign, we have a $UC_4$ and must merge the temporary colors $i$ and $k$. 
Suppose they have different signs. Note that $y$ and $z$ (resp. $u$) differ only by their $j$th coordinate (resp. $k$th coordinate).
Let $a$ be the vertex with the same coordinate as $u$ except for its $k$th coordinate which is equal to the $k$th coordinate of $z$ (see Figure~\ref{fig:algo-case-2}). Note that $a$ appears before $z$ and $u$ in the BFS ordering. Since the vertex $a$ is a neighbor of $z$ and $u$, both edges $za$ and $ua$ are down-edges of $a$. 
Hence $za \in Treated$ and $za$ has the same sign as $xx_u$ which is different from the sign of $uy$, and $yz$ and $ua$ also have the same sign since both are in $Treated$. 
In particular $yuaz$ is a $UC_4$ and these four vertices must be in the same factor of $P$. This implies that we must merge the temporary colors $j$ and $k$ which implies merging $i$ and $k$.


\medskip

At the end $J$ is finer than $P$ and $J$ is an $s$-decomposition by Claim~\ref{claim:treated}. Hence $J = P$.

\medskip

\noindent\textbf{Complexity:} Due to the similarity of our algorithm with the one in \cite{Imrich2018}, most of the complexity arguments given in \cite{Imrich2018} are still valid for our algorithm. The only differences between the two algorithms are the presence of the three sets $Done$, $S$ and $Treated$, two more if blocks and the need to switch at some vertices. Let us address these three points. Each set can be encoded by a boolean in the data structure of vertices/edges. The second for loop  checks each edge $xy$ twice, once for each endpoint, but this still amounts to a $O(m)$ iteration of the loop. The two additional if blocks are a $O(1)$ overhead for each iteration of the loop. The switch operation is another $O(m)$ total overhead as each edge can be switched at most once thanks to the presence of the set $S$. Hence the algorithm runs in time $O(m)$. The reader can find more details in \cite{Imrich2018}, and in particular, how to compute the projections in constant time.
\end{proof}

Note that this algorithm not only computes the prime $s$-decomposition of $(G,\sigma)$ but finds a signature $\sigma' \equiv \sigma$ for which all layers of the Cartesian products have the same signature as their corresponding factors.

\section{Chromatic number of Cartesian products of complete signed graphs and upper bounds}
\label{sec:cartesian:complete}

In this section, we show a simple upper bound on the chromatic number of a Cartesian product of two signed graphs and compute the chromatic number of some special  complete signed graphs. We start by defining a useful tool on signed graphs.

\subsection{$s$-redundant sets}

In what follows we define the notion of an $s$-redundant set in a signed graph. Intuitively, if $S$ is an $s$-redundant set of $(G,\sigma)$ and $x$ and $y$ are two vertices  cannot be mapped to a same vertex by any homomorphism of $\G$, then they cannot be mapped to a same vertex by a homomorphism of $(G,\sigma) - S$.


\begin{definition}
\label{def:redundant}
Let $(G,\sigma)$ be a signed graph and $S \subseteq V(G)$. We say that the set $S$ is {\em $s$-redundant} if and only if, for every $x,y \in V(G) -S$ such that $xy \notin E(G)$, every $z \in S$ and every signature $\sigma'$ with $\sigma' \equiv \sigma$,  if $xzy = UP_3$ in  $(G,\sigma')$ then there exists $w \in V(G) -S$ such that $xwy= UP_3$ in  $(G,\sigma')$. 
\end{definition}

The following proposition provides an alternative formulation of the definition which is useful in order to prove that a set is an $s$-redundant set.

\begin{proposition}
\label{prop:redundant-BC4}
If $(G,\sigma)$ is a signed graph and $S \subseteq V(G)$, then $S$ is $s$-redundant if and only if for every $z \in S$, and every $ x,y\in N(z) \setminus S$ with $xy \notin E(G)$, there exists $w \in V(G) \setminus S$ such that $xwyz$ is a $BC_4$.
\end{proposition}

\begin{proof}
Take $x,y \in V(G) -S$ such that $xy \notin E(G)$ and $ z \in S$. If $xzy = UP_3$ in a signature $\sigma' \equiv \sigma$, then $x,y \in N(z)$. Now if $S$ is an $s$-redundant set, then with the notation of the definition $xzyw$ is a $BC_4$ in $(G,\sigma')$ and thus in $(G,\sigma)$. 
If $xzyw$ is a $BC_4$ and $xzy$ is a $UP_3$ in a given signature $\sigma'$, then $xwy$ is also a $UP_3$ as $xzyw$ is balanced. This proves the equivalence between the two statements.
\end{proof}

The next theorem is the reason why we defined this notion. It allows us to compute an upper bound on the chromatic number of a signed  graph as a function of the chromatic number of one of its subgraphs. One example of utilisation of this notion is given by the proof of Theorem~\ref{th:K+K-}.

\begin{theorem}
\label{thm:redundant}
If $(G,\sigma)$ is a signed graph and $S$ is an $s$-redundant set of $(G,\sigma)$, then
$$ \chi_s(G,\sigma) \leq \abs{S} + \chi_s((G,\sigma)-S).$$
\end{theorem}

\begin{proof}
Let $c$ be a coloring of a signed graph $(G,\sigma') -S$ with $\chi_s((G,\sigma)-S)$ colors where $(G,\sigma') \equiv (G,\sigma)$.
We define the coloring $c'$ of $(G,\sigma')$ as follows:  $c'(v) = c(v)$ when $v \notin S$ and $c'(v)$ is a new color when $v \in S$.
Hence $c'$ uses at most $\abs{S} + \chi_s((G,\sigma)-S)$ colors. 

It is left to show that it is indeed a coloring of $(G,\sigma')$.
As $c$ is a coloring, $c'$ does not assign the same color to two adjacent vertices. Suppose, by contradiction, that there exists two edges $xy$ and $x'y'$ of opposite sign such that $c'(x) = c'(x')$ and $c'(y) = c'(y')$. As $c$ is a coloring, all four vertices cannot be in $G-S$. W.l.o.g. suppose that $x \in S$. By definition of $c'$, $x' = x$, $y,y' \notin S$ and $yxy'$ is a $UP_3$ in $(G,\sigma')$. As $S$ is an $s$-redundant set, there exists $w \notin S$ such that $ywy'$ is a  $UP_3$ in $(G,\sigma') -S$. This contradicts the fact that $c$ is a coloring of $(G,\sigma') -S$.
%
%
%
\end{proof}

This result does not hold for any set $S$. For example, if $(G,\sigma)= UC_4$ and $S=\set{v}$ is a single vertex of $G$, then $\chi_s(G,\sigma) = 4$ and $\chi_s((G,\sigma)-v) =2$.

\subsection{Back to Cartesian products of complete signed graphs}
As a direct corollary of Theorem~\ref{thm:cartesianHomCompatibility}, we get the following upper bound on the chromatic number of a Cartesian product of signed graphs.

\begin{corollary}
\label{cor:UpperBoundCartesianProduct}
If $(G_1,\sigma_1)$, \dots, $(G_k,\sigma_k)$ are $k$ signed graphs, then:
$$ \chi_s((G_1,\sigma_1) \ \square\ \cdots\ \square\ (G_k,\sigma_k)) \leq \prod_{1\leq i \leq k}\chi_s(G_i,\sigma_i).$$ 
\end{corollary}

We consider the Cartesian product of balanced and antibalanced complete graphs in our next result. Recall that $K_p^+$ (\resp $K_q^-$) is the complete graph with only positive edges (\resp negative edges). 

\begin{theorem}
\label{th:K+K-}
For every two integers $p,q$ with $p, q \geq 2$, we have 
$$\chi_s(K_p^+\ \square\ K_q^-) = \ceil{\frac{pq}{2}} .$$ 
\end{theorem}

\begin{figure}[t]
\centering
\subfloat[Notation of the proof]{
\begin{tikzpicture}
\tikzstylemacro{}

\node[n] (11) at (1,1) {};
\node[n, label=below right:$a$] (12) at (1,2) {};
\node[n, label=below right:$x$] (13) at (1,3) {};
\node[n] (21) at (2,1) {};
\node[n, label=below right:$y$] (22) at (2,2) {};
\node[n] (23) at (2,3) {};
\node[n, label=below right:$z$] (31) at (3,1) {};
\node[n] (32) at (3,2) {};
\node[n] (33) at (3,3) {};

\foreach \i in {1,2,3}{
	\draw[r] (\i1) -- (\i2) -- (\i3);
	\draw[r] (\i1) edge[bend left] (\i3);
}

\foreach \j in {1,2,3}{
	\draw[g] (1\j) -- (2\j);
	\draw[g] (2\j) -- (3\j);
	\draw[g] (1\j) edge[bend left] (3\j);
}

\end{tikzpicture}
}\hspace{2cm}
\subfloat[A coloring of $(H,\sigma)$ with $5$ colors]{
\begin{tikzpicture}
\tikzstylemacro{}

\node[ns, label=below right:$1$] (11) at (1,1) {};
\node[ns, label=below right:$2$] (12) at (1,2) {};
\node[ns, label=below right:$3$] (13) at (1,3) {};
\node[ns, label=below right:$4$] (21) at (2,1) {};
\node[n, label=below right:$1$] (22) at (2,2) {};
\node[n, label=below right:$2$] (23) at (2,3) {};
\node[n, label=below right:$3$] (31) at (3,1) {};
\node[n, label=below right:$5$] (32) at (3,2) {};
\node[n, label=below right:$4$] (33) at (3,3) {};

\foreach \i in {1,3}{
	\draw[r] (\i1) -- (\i2) -- (\i3);
	\draw[r] (\i1) edge[bend left] (\i3);
}

\draw[g] (21) -- (22);
\draw[r] (22) -- (23);
\draw[g] (21) edge[bend left] (23);

\foreach \j in {2,3}{
	\draw[r] (1\j) -- (2\j);
	\draw[g] (2\j) -- (3\j);
	\draw[r] (1\j) edge[bend left] (3\j);
}
\draw[g] (11) -- (21);
\draw[r] (21) -- (31);
\draw[r] (11) edge[bend left] (31);
\end{tikzpicture}
}
\caption{The signed graph $(H,\sigma) = K_3^+ \ssquare K_3^-$ of Theorem~\ref{th:K+K-}. The big squared vertices have been switched.}
\label{fig:K+K-}
\end{figure}
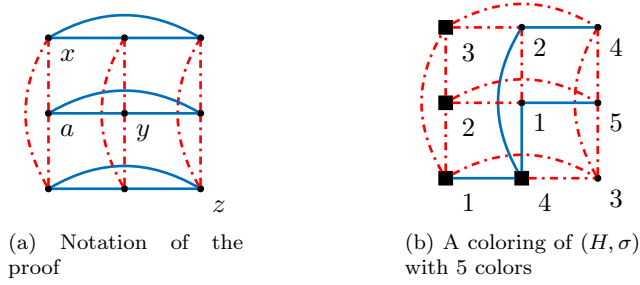

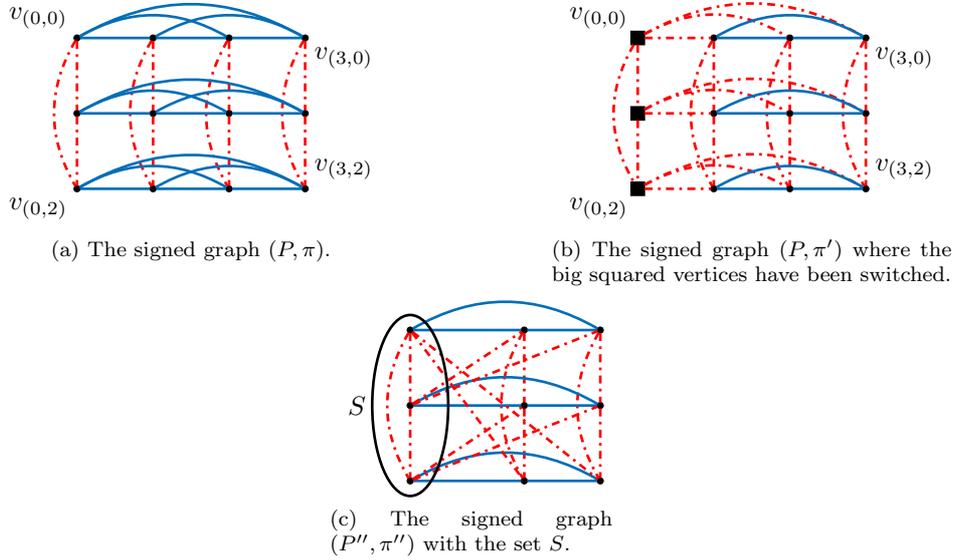
\begin{figure}[t]
\centering
\subfloat[The signed graph $(P,\pi)$.\label{fig:p}]{
\begin{tikzpicture}
\tikzstylemacro{}

\foreach \i in {1,2,3,4}{
	\foreach \j in {1,2,3}{
    	\node[n] (\i\j) at (\i,3-\j) {};
    }
}
\coordinate[label=above left:$v_{(0,0)}$] (i) at (1,2) {};
\coordinate[label=below right:$v_{(3,0)}$] (i) at (4,2) {};
\coordinate[label=below left:$v_{(0,2)}$] (i) at (1,0) {};
\coordinate[label=above right:$v_{(3,2)}$] (i) at (4,0) {};

\foreach \i in {1,2,3,4}{
	\draw[r] (\i1) -- (\i2) -- (\i3);
	\draw[r] (\i1) edge[bend right] (\i3);
}

\foreach \j in {1,2,3}{
	\draw[g] (1\j) -- (2\j);
	\draw[g] (2\j) -- (3\j);
	\draw[g] (3\j) -- (4\j);
	\draw[g] (1\j) edge[bend left] (3\j);
	\draw[g] (1\j) edge[bend left] (4\j);
	\draw[g] (2\j) edge[bend left] (4\j);
}
\end{tikzpicture}
}\hspace{2cm}
\subfloat[The signed graph $(P,\pi')$ where the big squared vertices have been switched.\label{fig:pprime}]{
\begin{tikzpicture}
\tikzstylemacro{}

\foreach \j in {1,2,3}{
    \node[ns] (1\j) at (1,3-\j) {};
}

\foreach \i in {2,3,4}{
	\foreach \j in {1,2,3}{
    	\node[n] (\i\j) at (\i,3-\j) {};
    }
}
\coordinate[label=above left:$v_{(0,0)}$] (i) at (1,2) {};
\coordinate[label=below right:$v_{(3,0)}$] (i) at (4,2) {};
\coordinate[label=below left:$v_{(0,2)}$] (i) at (1,0) {};
\coordinate[label=above right:$v_{(3,2)}$] (i) at (4,0) {};

\foreach \i in {1,2,3,4}{
	\draw[r] (\i1) -- (\i2) -- (\i3);
	\draw[r] (\i1) edge[bend right] (\i3);
}

\foreach \j in {1,2,3}{
	\draw[r] (1\j) -- (2\j);
	\draw[g] (2\j) -- (3\j);
	\draw[g] (3\j) -- (4\j);
	\draw[r] (1\j) edge[bend left] (3\j);
	\draw[r] (1\j) edge[bend left] (4\j);
	\draw[g] (2\j) edge[bend left] (4\j);
}

\end{tikzpicture}
}\hspace{2cm}
\subfloat[The signed graph $(P'',\pi'')$ with the set $S$.\label{fig:psecond}]{
\begin{tikzpicture}
\tikzstylemacro{}

\foreach \j in {1,2,3}{
	\node[n] (2\j) at (1.5,3-\j) {};
}

\foreach \i in {3,4}{
	\foreach \j in {1,2,3}{
    	\node[n] (\i\j) at (\i,3-\j) {};
    }
}

\foreach \i in {2,3,4}{
	\draw[r] (\i1) -- (\i2) -- (\i3);
	\draw[r] (\i1) edge[bend right] (\i3);
}

\foreach \j in {1,2,3}{
	\draw[g] (2\j) -- (3\j);
	\draw[g] (3\j) -- (4\j);
	\draw[g] (2\j) edge[bend left] (4\j);
}
\draw[r] (22) -- (31);
\draw[r] (22) -- (41);

\draw[r] (23) -- (32);
\draw[r] (23) -- (42);

\draw[r] (21) -- (33);
\draw[r] (21) -- (43);

\draw[line width=1pt] (1.5,1) ellipse (0.5cm and 1.2cm);
\node (i) at (0.8,1) {$S$};

\end{tikzpicture}
}
\caption{The signed graphs $(P,\pi)$, $(P,\pi')$ and $(P'',\pi'')$ of Theorem~\ref{th:K+K-} when $(P,\pi) = K_4^+ \ssquare K_3^-$.}
\label{fig:K+K-2}
\end{figure}

\begin{proof}
Let us note $ (P,\pi) = K_p^+\ \square\ K_q^-$. By symmetry between the sets of positive and negative edges, we can suppose $p\geq q$.
First let us show that $\chi_s(P,\pi) \geq \ceil{\frac{pq}{2}}$. 

Suppose it is not the case. Let $\varphi$ be an optimal homomorphism of $(P,\pi)$. By the pigeon hole principle, there exist $x$, $y$ and $z$ three vertices of the Cartesian product with the same image by $\varphi$. They belong to three distinct  $K_p^+$-layers and three distinct $K_q^-$-layer as these are complete graphs. Consider the subgraph $(H,\sigma)$ of $(P,\pi)$ composed of vertices which are in the same $K_p^+$-layers as one of $x$, $y$, $z$ and in the same $K_q^+$-layers as one of $x$, $y$ and $z$. We have $(H,\sigma) = K_3^+\ssquare K_3^-$ (see Figure~\ref{fig:K+K-}).

By assumption $x$, $y$ and $z$ of $(H,\sigma)$ are identified by $\varphi$ (possibly after switching some of them). By the pigeon hole principle, two of $x$, $y$ and $z$ are both switched or both non-switched. Without loss of generality suppose they are $x$ and $y$. Then if $a$ is one of their common neighbors in $H$, the edges $xa$ and $ya$ are of different signs, thus $x$ and $y$ cannot be identified. This is a contradiction.

\medskip

We now prove that $\chi_s(P,\pi) \leq \ceil{\frac{pq}{2}}$ by induction. 
If $p=2$, then $(P,\pi) \equiv BC_4$ and $\chi_s(P,\pi) = 2 \leq 2$. If $p=3$ and $q=2$, then $(P,\pi) \equiv BC_3 \ssquare K_2$ whose chromatic number is $3$. If $p=3$ and $q=3$, then $(P,\pi) \equiv K_3^+\ssquare K_3^-$. In this case, we have $\chi_s(P,\pi) = 5$, as Figure~\ref{fig:K+K-} gives a $5$-coloring of $(P,\pi)$.

Now we can assume $p\geq 4$. 
Let  $V(P) = \set{v_{(i,j)}, 0\leq i < p, 0\leq j < q}$ such that for every $i$, the set $\set{v_{(i,j)}}_{0\leq j < q}$ induces a negative complete graph and for every $j$, the set $\set{v_{(i,j)}}_{0\leq i < p}$ induces a positive complete graph (see Figure~\ref{fig:p}). Now switch all vertices in $\setcond{v_{(i,j)}}{ i =0}$ to obtain the signed graph $(P,\pi')$ (see Figure~\ref{fig:pprime}) and then identify $v_{(0,j)}$ with $v_{(1,j+1)}$ (which are non adjacent) for every $j \in \setk{0,q-1}$, where indices are taken modulo $q$, to obtain the graph $(P'',\pi'')$ (see Figure~\ref{fig:psecond}). 
Let $S$ be the set of identified vertices in $(P',\pi')$. We want to show that $S$ is $s$-redundant in order to use the induction hypothesis. Take $z \in S$ and $x,y \in N(z) \setminus S$ such that $xy \notin E(P'')$. If $xzy$ is an unbalanced path of length $2$, then $x$ is some $v_{(i,j)}$ and $y$ is some $v_{(k,j+1)}$ with $i,k \geq 2$. For $a = v_{(i,j+1)}$, $xayz$ is a $BC_4$. 

By Proposition~\ref{prop:redundant-BC4}, $S$ is $s$-redundant and thus $$\chi_s(P,\pi) \leq \chi_s(P'',\pi'') \leq \abs{S} + \chi_s((P',\pi') - S)$$ by Theorem~\ref{thm:redundant}. 
By induction hypothesis, as $(P'',\pi'') - S = K_{p-2}^+ \ssquare K_q^-$, we get $\chi_s((P'',\pi'') - S) \leq \ceil{\frac{(p-2)q}{2}}$. Thus $\chi_s(P,\pi) \leq q + \ceil{\frac{pq}{2}} - q \leq \ceil{\frac{pq}{2}}$.
\end{proof}

Note that $\chi_s(K_p^+\ssquare K_p^-) = O(\Delta^2)$ where $\Delta$ is the maximum degree of $K_p^+\ssquare K_p^-$ (\ie $\Delta = 2p-1)$. Indeed, the chromatic number is $O(p^2)$ while $\Delta^2 = (2p-1)^2 = 4p^2 -4p +1 $. 
Also, for this Cartesian product, the upper bound of Corollary~\ref{cor:UpperBoundCartesianProduct} is $p^2$ while we proved in Theorem~\ref{th:K+K-} that the chromatic number is $\ceil{\frac{p^2}{2}}$. We thus have an example where the chromatic number is greater than half the simple upper bound.
\begin{question}
What is the supremum of the set of real numbers $\lambda \in [\frac{1}{2},1]$ such that there exist signed graphs $(G_1,\sigma_1),\dots, (G_k,\sigma_k)$, each with at least one edge, such that: $$ \chi_s((G_1,\sigma_1) \ \square\ \cdots\ \square\ (G_k,\sigma_k)) \leq \lambda \prod_{1\leq i \leq k}\chi_s(G_i,\sigma_i)?$$
\end{question}

\begin{figure}[t]
\centering
\begin{tikzpicture}[scale=0.03000]
\tikzstyle{n} = [draw,circle,minimum size=0.8mm,inner sep=0pt,outer sep=0pt,fill=black];
\tikzstyle{b} = [line width=1pt, color=NavyBlue]; 
\tikzstyle{r} = [line width=1pt, color=red, dashed]; 
\node[n] (v0) at (200.00,0.00) {};
\node[n] (v1) at (187.00,68.00) {};
\node[n] (v2) at (153.00,128.00) {};
\node[n] (v3) at (100.00,173.00) {};
\node[n] (v4) at (34.00,196.00) {};
\node[n] (v5) at (-34.00,196.00) {};
\node[n] (v6) at (-99.00,173.00) {};
\node[n] (v7) at (-153.00,128.00) {};
\node[n] (v8) at (-187.00,68.00) {};
\node[n] (v9) at (-200.00,0.00) {};
\node[n] (v10) at (-187.00,-68.00) {};
\node[n] (v11) at (-153.00,-128.00) {};
\node[n] (v12) at (-100.00,-173.00) {};
\node[n] (v13) at (-34.00,-196.00) {};
\node[n] (v14) at (34.00,-196.00) {};
\node[n] (v15) at (99.00,-173.00) {};
\node[n] (v16) at (153.00,-128.00) {};
\node[n] (v17) at (187.00,-68.00) {};
\draw (v0) edge[b] (v1);
\draw (v0) edge[b] (v2);
\draw (v0) edge[b] (v3);
\draw (v0) edge[b] (v4);
\draw (v0) edge[r] (v5);
\draw (v0) edge[b] (v6);
\draw (v0) edge[b] (v7);
\draw (v0) edge[r] (v8);
\draw (v0) edge[b] (v9);
\draw (v0) edge[b] (v10);
\draw (v0) edge[b] (v11);
\draw (v0) edge[b] (v12);
\draw (v0) edge[r] (v13);
\draw (v0) edge[b] (v14);
\draw (v0) edge[b] (v15);
\draw (v0) edge[r] (v16);
\draw (v0) edge[r] (v17);
\draw (v1) edge[b] (v2);
\draw (v1) edge[b] (v3);
\draw (v1) edge[b] (v4);
\draw (v1) edge[b] (v5);
\draw (v1) edge[r] (v6);
\draw (v1) edge[r] (v7);
\draw (v1) edge[b] (v8);
\draw (v1) edge[b] (v9);
\draw (v1) edge[r] (v10);
\draw (v1) edge[r] (v11);
\draw (v1) edge[b] (v12);
\draw (v1) edge[b] (v13);
\draw (v1) edge[b] (v14);
\draw (v1) edge[r] (v15);
\draw (v1) edge[r] (v16);
\draw (v1) edge[b] (v17);
\draw (v2) edge[b] (v3);
\draw (v2) edge[r] (v4);
\draw (v2) edge[r] (v5);
\draw (v2) edge[b] (v6);
\draw (v2) edge[b] (v7);
\draw (v2) edge[b] (v8);
\draw (v2) edge[b] (v9);
\draw (v2) edge[r] (v10);
\draw (v2) edge[r] (v11);
\draw (v2) edge[r] (v12);
\draw (v2) edge[r] (v13);
\draw (v2) edge[b] (v14);
\draw (v2) edge[r] (v15);
\draw (v2) edge[b] (v16);
\draw (v2) edge[r] (v17);
\draw (v3) edge[b] (v4);
\draw (v3) edge[b] (v5);
\draw (v3) edge[r] (v6);
\draw (v3) edge[b] (v7);
\draw (v3) edge[b] (v8);
\draw (v3) edge[r] (v9);
\draw (v3) edge[r] (v10);
\draw (v3) edge[b] (v11);
\draw (v3) edge[b] (v12);
\draw (v3) edge[r] (v13);
\draw (v3) edge[r] (v14);
\draw (v3) edge[b] (v15);
\draw (v3) edge[r] (v16);
\draw (v3) edge[r] (v17);
\draw (v4) edge[r] (v5);
\draw (v4) edge[r] (v6);
\draw (v4) edge[b] (v7);
\draw (v4) edge[b] (v8);
\draw (v4) edge[b] (v9);
\draw (v4) edge[r] (v10);
\draw (v4) edge[b] (v11);
\draw (v4) edge[r] (v12);
\draw (v4) edge[b] (v13);
\draw (v4) edge[b] (v14);
\draw (v4) edge[b] (v15);
\draw (v4) edge[b] (v16);
\draw (v4) edge[b] (v17);
\draw (v5) edge[r] (v6);
\draw (v5) edge[r] (v7);
\draw (v5) edge[r] (v8);
\draw (v5) edge[r] (v9);
\draw (v5) edge[b] (v10);
\draw (v5) edge[r] (v11);
\draw (v5) edge[r] (v12);
\draw (v5) edge[r] (v13);
\draw (v5) edge[b] (v14);
\draw (v5) edge[b] (v15);
\draw (v5) edge[b] (v16);
\draw (v5) edge[r] (v17);
\draw (v6) edge[b] (v7);
\draw (v6) edge[b] (v8);
\draw (v6) edge[r] (v9);
\draw (v6) edge[r] (v10);
\draw (v6) edge[b] (v11);
\draw (v6) edge[r] (v12);
\draw (v6) edge[r] (v13);
\draw (v6) edge[r] (v14);
\draw (v6) edge[b] (v15);
\draw (v6) edge[r] (v16);
\draw (v6) edge[b] (v17);
\draw (v7) edge[b] (v8);
\draw (v7) edge[b] (v9);
\draw (v7) edge[b] (v10);
\draw (v7) edge[r] (v11);
\draw (v7) edge[r] (v12);
\draw (v7) edge[b] (v13);
\draw (v7) edge[b] (v14);
\draw (v7) edge[r] (v15);
\draw (v7) edge[b] (v16);
\draw (v7) edge[r] (v17);
\draw (v8) edge[r] (v9);
\draw (v8) edge[r] (v10);
\draw (v8) edge[b] (v11);
\draw (v8) edge[b] (v12);
\draw (v8) edge[r] (v13);
\draw (v8) edge[r] (v14);
\draw (v8) edge[b] (v15);
\draw (v8) edge[b] (v16);
\draw (v8) edge[b] (v17);
\draw (v9) edge[r] (v10);
\draw (v9) edge[r] (v11);
\draw (v9) edge[b] (v12);
\draw (v9) edge[b] (v13);
\draw (v9) edge[b] (v14);
\draw (v9) edge[r] (v15);
\draw (v9) edge[b] (v16);
\draw (v9) edge[b] (v17);
\draw (v10) edge[r] (v11);
\draw (v10) edge[b] (v12);
\draw (v10) edge[r] (v13);
\draw (v10) edge[r] (v14);
\draw (v10) edge[b] (v15);
\draw (v10) edge[b] (v16);
\draw (v10) edge[b] (v17);
\draw (v11) edge[b] (v12);
\draw (v11) edge[b] (v13);
\draw (v11) edge[r] (v14);
\draw (v11) edge[b] (v15);
\draw (v11) edge[r] (v16);
\draw (v11) edge[b] (v17);
\draw (v12) edge[r] (v13);
\draw (v12) edge[r] (v14);
\draw (v12) edge[b] (v15);
\draw (v12) edge[r] (v16);
\draw (v12) edge[b] (v17);
\draw (v13) edge[r] (v14);
\draw (v13) edge[r] (v15);
\draw (v13) edge[b] (v16);
\draw (v13) edge[r] (v17);
\draw (v14) edge[r] (v15);
\draw (v14) edge[b] (v16);
\draw (v14) edge[r] (v17);
\draw (v15) edge[b] (v16);
\draw (v15) edge[r] (v17);
\draw (v16) edge[b] (v17);
\end{tikzpicture}
\caption{A signed graph $K$ of order $18$ such that $\chi_s(K\ssquare K_2) = 25$.}
\label{fig:K18+7}
\end{figure}
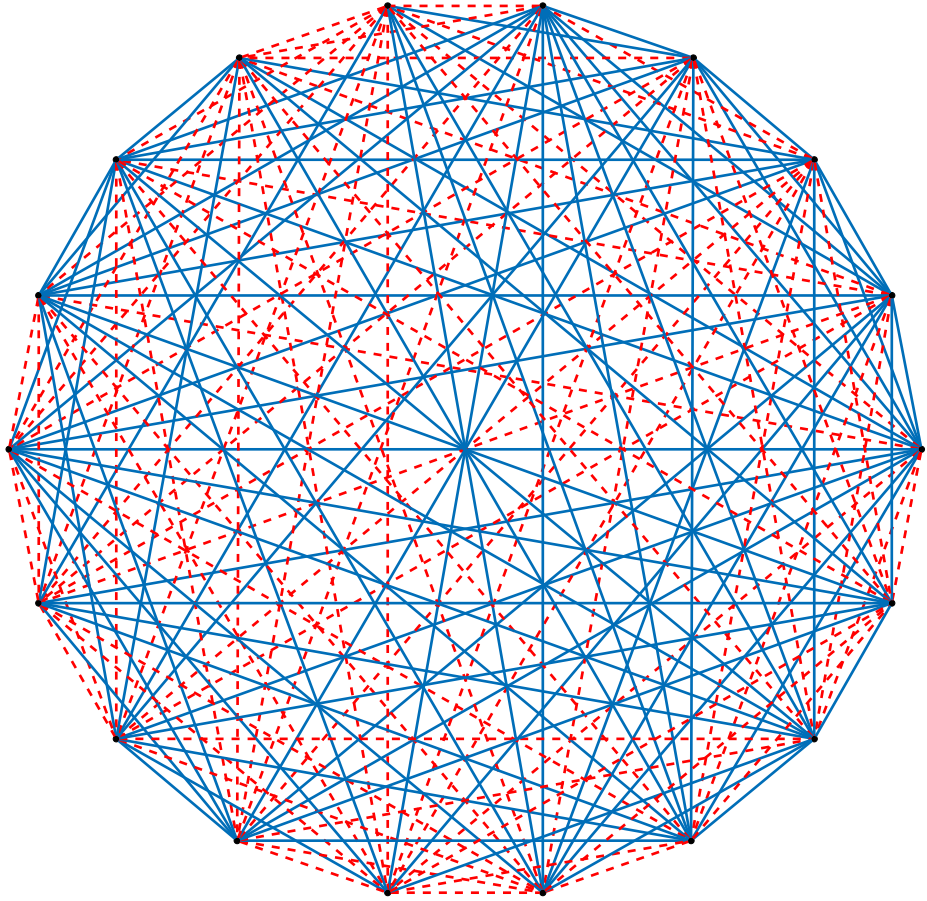

 In Figure~\ref{fig:K18+7}, we have an example of a graph $K$ such that $K \ssquare K_2$ has chromatic number $25$ (checked by computer). The ratio between the chromatic number and the upper bound is $\frac{25}{36} = 0.69444$. It is the largest ratio we have found by randomly sampling bigger and bigger complete signed graphs. This leads us to believe that the following conjecture holds.
 
 \begin{conj}
For every fixed $\varepsilon > 0$, there exist signed graphs $(G_1,\sigma_1),\dots, (G_k,\sigma_k)$, with each at least one edge, such that: $$ \chi_s((G_1,\sigma_1) \ \square\ \cdots\ \square\ (G_k,\sigma_k)) \geq (1-\varepsilon) \cdot \prod_{1\leq i \leq k}\chi_s(G_i,\sigma_i).$$
\end{conj}

\section{Chromatic number of Cartesian products of signed cycles}
\label{sec:cartesian:cycles}

The goal of this section is to determine the chromatic number of the Cartesian product of two signed cycles. As there are four kind of cycles (balanced/unbalanced and even/odd length), we have a number of cases to analyse. In most cases some simple observations are sufficient to conclude. For the other cases, we need the following lemma whose proof is given in subsections \ref{sec:cycles:strartproof} to  \ref{sec:proof-ending},  due to its length.

\begin{lemma}
\label{lem:ProdCyclePas4}
For every two integers $p$,$q \in \NN$:
$$ \chi_s(UC_q\ssquare BC_{2p+1} ) > 4.$$
\end{lemma}

With this lemma, we can state the main result of this section.

\begin{theorem}\label{thm:cycles}
If $(C_1,\sigma)$ and $(C_2,\sigma_2)$ are two signed cycles, then the chromatic number of $(P,\pi) = (C_1,\sigma_1)\ssquare (C_2,\sigma_2)$ is given by Table~\ref{fig:table_chi_cycle_prod}, depending on the types of $(C_1,\sigma_1)$ and $(C_2,\sigma_2)$.
\end{theorem}

\begin{table}[H]
\centering
\begin{tabular}{|l|l|l|l|l|}
 \hline
 $(C_1,\sigma_1)\ssquare (C_2,\sigma_2)$ & $BC_{even}$ & $BC_{odd}$ & $UC_{even}$ & $UC_{odd}$ \\
 \hline
 $BC_{even}$ & 2 & 3 & 4 & 3 \\
 \hline
 $BC_{odd}$ & 3 & 3 & 5 & 5 \\
 \hline
 $UC_{even}$ & 4 & 5 & 4 & 5 \\
 \hline
 $UC_{odd}$ & 3 & 5 & 5 & 3 \\
 \hline
\end{tabular}
\caption{The chromatic number of Cartesian products of signed cycles.}
\label{fig:table_chi_cycle_prod}
\end{table}


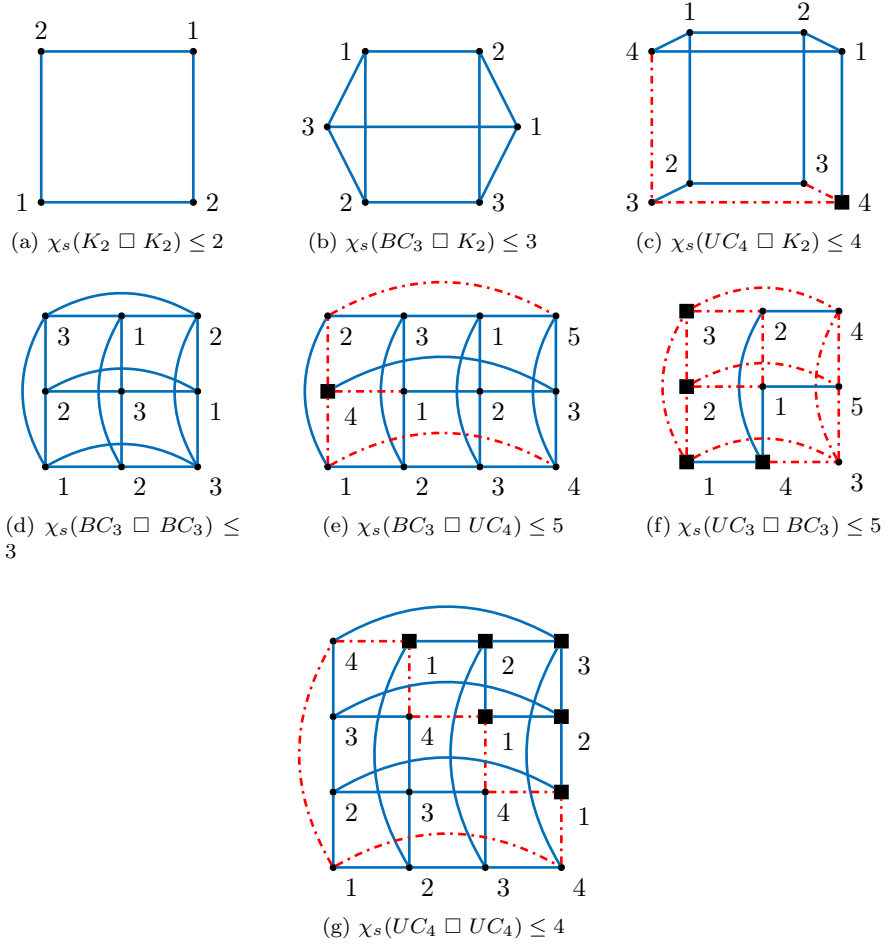
\begin{figure}[!t]
\centering
\subfloat[$\chi_s(K_2\ssquare K_2) \leq 2$]{
\begin{tikzpicture}
\tikzstylemacro{}

\node[n, label=left:$1$] (C) at (0,-1) {};
\node[n, label=above:$2$] (A) at  (0,1) {};
\node[n, label=right:$2$] (D) at (2,-1) {};
\node[n, label=above:$1$] (B) at (2,1) {};
\draw[g] (A)--(B);
\draw[g] (A)--(C)--(D)--(B);
\end{tikzpicture}
}\hspace{.5cm}
\subfloat[$\chi_s(BC_3\ssquare K_2) \leq 3$]{
\begin{tikzpicture}
\tikzstylemacro{}

\node[n, label=left:$1$] (A1) at (0,1) {};
\node[n, label=left:$2$] (A2) at  (0,-1) {};
\node[n, label=left:$3$] (A3) at (-.5,0) {};
\node[n, label=right:$2$] (B1) at (1.5,1) {};
\node[n, label=right:$3$] (B2) at  (1.5,-1) {};
\node[n, label=right:$1$] (B3) at (2,0) {};
\draw[g] (A1)--(A2) -- (A3) -- (A1);
\draw[g] (B1)--(B2) -- (B3) -- (B1);
\foreach \i in {1,2,3}{
	\draw[g] (A\i)--(B\i);
}
\end{tikzpicture}
}\hspace{.5cm}
\subfloat[$\chi_s(UC_4\ssquare K_2) \leq 4$]{
\begin{tikzpicture}
\tikzstylemacro{}

\node[n, label=above:$1$] (A1) at (0,1) {};
\node[n, label=above left:$2$] (A2) at  (0,-1) {};
\node[n, label=left:$3$] (A3) at (-.5,-1.25) {};
\node[n, label=left:$4$] (A4) at (-.5,.75) {};
\node[n, label=above:$2$] (B1) at (1.5,1) {};
\node[n, label=above right:$3$] (B2) at  (1.5,-1) {};
\node[ns, label=right:$4$] (B3) at (2,-1.25) {};
\node[n, label=right:$1$] (B4) at (2,.75) {};
\draw[g] (A4) --(A1)--(A2) -- (A3); 
\draw[r] (A4) -- (A3);
\draw[g] (B4) --(B1)-- (B2);
\draw[r] (B2) -- (B3); 
\draw[g] (B4) -- (B3);
\foreach \i in {1,2,4}{
	\draw[g] (A\i)--(B\i);
}
\draw[r] (A3) -- (B3);
\end{tikzpicture}
}
%
%
\\
\subfloat[$\chi_s(BC_3\ssquare BC_3) \leq 3$]{
\begin{tikzpicture}
\tikzstylemacro{}

\node[n, label=below right:$1$] (11) at (1,1) {};
\node[n, label=below right:$2$] (12) at (1,2) {};
\node[n, label=below right:$3$] (13) at (1,3) {};
\node[n, label=below right:$2$] (21) at (2,1) {};
\node[n, label=below right:$3$] (22) at (2,2) {};
\node[n, label=below right:$1$] (23) at (2,3) {};
\node[n, label=below right:$3$] (31) at (3,1) {};
\node[n, label=below right:$1$] (32) at (3,2) {};
\node[n, label=below right:$2$] (33) at (3,3) {};

\foreach \i in {1,2,3}{
	\draw[g] (\i1) -- (\i2) -- (\i3);
	\draw[g] (\i1) edge[bend left] (\i3);
}

\foreach \j in {1,2,3}{
	\draw[g] (1\j) -- (2\j) -- (3\j);
	\draw[g] (1\j) edge[bend left] (3\j);
}
\end{tikzpicture}
}\hspace{0.5cm}
\subfloat[$\chi_s(BC_3\ssquare UC_4) \leq 5$]{
\begin{tikzpicture}
\tikzstylemacro{}

\node[n, label=below right:$1$] (11) at (1,1) {};
\node[ns, label=below right:$4$] (12) at (1,2) {};
\node[n, label=below right:$2$] (13) at (1,3) {};
\node[n, label=below right:$2$] (21) at (2,1) {};
\node[n, label=below right:$1$] (22) at (2,2) {};
\node[n, label=below right:$3$] (23) at (2,3) {};
\node[n, label=below right:$3$] (31) at (3,1) {};
\node[n, label=below right:$2$] (32) at (3,2) {};
\node[n, label=below right:$1$] (33) at (3,3) {};
\node[n, label=below right:$4$] (41) at (4,1) {};
\node[n, label=below right:$3$] (42) at (4,2) {};
\node[n, label=below right:$5$] (43) at (4,3) {};

\foreach \i in {2,3,4}{
	\draw[g] (\i1) -- (\i2) -- (\i3);
	\draw[g] (\i1) edge[bend left] (\i3);
}
\draw[r] (11) -- (12) -- (13);
\draw[g] (11) edge[bend left] (13);

\foreach \j in {1,3}{
	\draw[g] (1\j) -- (2\j) -- (3\j) -- (4\j);
	\draw[r] (1\j) edge[bend left] (4\j);
}
\draw[r] (12) -- (22);
\draw[g] (22) -- (32) -- (42);
\draw[g] (12) edge[bend left] (42);
\end{tikzpicture}
}\hspace{0.5cm}
\subfloat[$\chi_s(UC_3\ssquare BC_3) \leq 5$]{
\begin{tikzpicture}
\tikzstylemacro{}

\node[ns, label=below right:$1$] (11) at (1,1) {};
\node[ns, label=below right:$2$] (12) at (1,2) {};
\node[ns, label=below right:$3$] (13) at (1,3) {};
\node[ns, label=below right:$4$] (21) at (2,1) {};
\node[n, label=below right:$1$] (22) at (2,2) {};
\node[n, label=below right:$2$] (23) at (2,3) {};
\node[n, label=below right:$3$] (31) at (3,1) {};
\node[n, label=below right:$5$] (32) at (3,2) {};
\node[n, label=below right:$4$] (33) at (3,3) {};

\foreach \i in {1,3}{
	\draw[r] (\i1) -- (\i2) -- (\i3);
	\draw[r] (\i1) edge[bend left] (\i3);
}

\draw[g] (21) -- (22);
\draw[r] (22) -- (23);
\draw[g] (21) edge[bend left] (23);

\foreach \j in {2,3}{
	\draw[r] (1\j) -- (2\j);
	\draw[g] (2\j) -- (3\j);
	\draw[r] (1\j) edge[bend left] (3\j);
}
\draw[g] (11) -- (21);
\draw[r] (21) -- (31);
\draw[r] (11) edge[bend left] (31);
\end{tikzpicture}
}\\
\subfloat[$\chi_s(UC_4\ssquare UC_4) \leq 4$]{
\begin{tikzpicture}
\tikzstylemacro{}

\node[n, label=below right:$1$] (11) at (1,1) {};
\node[n, label=below right:$2$] (12) at (1,2) {};
\node[n, label=below right:$3$] (13) at (1,3) {};
\node[n, label=below right:$4$] (14) at (1,4) {};
\node[n, label=below right:$2$] (21) at (2,1) {};
\node[n, label=below right:$3$] (22) at (2,2) {};
\node[n, label=below right:$4$] (23) at (2,3) {};
\node[ns, label=below right:$1$] (24) at (2,4) {};
\node[n, label=below right:$3$] (31) at (3,1) {};
\node[n, label=below right:$4$] (32) at (3,2) {};
\node[ns, label=below right:$1$] (33) at (3,3) {};
\node[ns, label=below right:$2$] (34) at (3,4) {};
\node[n, label=below right:$4$] (41) at (4,1) {};
\node[ns, label=below right:$1$] (42) at (4,2) {};
\node[ns, label=below right:$2$] (43) at (4,3) {};
\node[ns, label=below right:$3$] (44) at (4,4) {};

\draw[g] (11) -- (12) -- (13) -- (14);
\draw[r] (11) edge[bend left] (14);
\draw[g] (21) -- (22) -- (23);
\draw[r] (23) -- (24);
\draw[g] (21) edge[bend left] (24);
\draw[g] (31) -- (32);
\draw[r] (32) -- (33);
\draw[g] (33) -- (34);
\draw[g] (31) edge[bend left] (34);
\draw[r] (41) -- (42);
\draw[g] (42) -- (43) -- (44);
\draw[g] (41) edge[bend left] (44);

\draw[g] (11) -- (21) -- (31) -- (41);
\draw[r] (11) edge[bend left] (41);
\draw[g] (12) -- (22) -- (32);
\draw[r] (32) -- (42);
\draw[g] (12) edge[bend left] (42);
\draw[g] (13) -- (23);
\draw[r] (23) -- (33);
\draw[g] (33) -- (43);
\draw[g] (13) edge[bend left] (43);
\draw[r] (14) -- (24);
\draw[g] (24) -- (34) -- (44);
\draw[g] (14) edge[bend left] (44);

\end{tikzpicture}
}
\caption{Coloring of Cartesian products of signed cycles. The large squared vertices have been switched in the Cartesian product. }
\label{fig:CasesProdCycles}
\end{figure}

\begin{proof}
If $G$ is a cycle of type $BC_{even}$ (resp. $BC_{odd}$, $UC_{even}$, $UC_{odd}$), then $G \sto BC_2 = K_2$ (resp. $BC_{3}$, $UC_{4}$, $UC_{3}$). By computing the chromatic numbers of the Cartesian products of  $(G,\sigma)$ and $(H,\pi)$ when they belong to $\set{K_2, BC_3, UC_4, UC_3}$, we get an upper bound for each of the Cartesian product type equal to the corresponding value in the table. These cases, up to symmetry between the sets of positives and negatives edges, are represented in Figure~\ref{fig:CasesProdCycles}. Note that to color some graphs, we switched some vertices.

For the lower bound, note that $\chi_s((C_1,\sigma_1)\ssquare (C_2,\sigma_2)) \geq \max(\chi_s(C_1,\sigma_1),\chi_s(C_2,\sigma_2))$. Theorem~\ref{thm:cyclespasproduit} concludes for the cases where the chromatic number is at most $4$. Lemma~\ref{lem:ProdCyclePas4} allows us to conclude for the remaining cases as $\chi_s(UC_q\ssquare BC_{2p+1}) = \chi_s(UC_q\ssquare UC_{2p+1})$ by symmetry between the two edge types.
\end{proof}

%

One further question would be to compute the chromatic number of the Cartesian product of an arbitrary number of signed cycles. Note that $BC_3 \ssquare BC_3 \sto BC_3$, and that the same holds for $K_2$, $UC_3$ and $UC_4$. This implies that, for these four graphs, it is only interesting to look at Cartesian products of the form $K_2^a \ssquare BC_3^b \ssquare UC_4^c \ssquare UC_3^d$ where $a,b,c,d \in \set{0,1}$. Moreover, we can suppose that $a =0$ if one of $b$, $c$ or $d$ is non zero. Thus the only interesting case left to solve is $\chi_s(BC_3 \ssquare UC_3 \ssquare UC_4)$.

To extend this to any length, using the same argument as in Theorem \ref{thm:cycles}, would require that we obtain a lower bound for $\chi_s(BC_{2p+1} \ssquare UC_{2q+1} \ssquare UC_{2r})$ equal to $\chi_s(BC_3 \ssquare UC_3 \ssquare UC_4)$.


\subsection{Definitions and preliminary results} \label{sec:cycles:strartproof}

We start by recalling some more definitions.

%
A graph $G$ is {\em bipartite} if we can partition $V(G)$ into $A \uplus B$ (where $\uplus$ is the disjoint union) such that every edge $xy$ of $G$ has one endpoint in $A$ and one endpoint in $B$.

For a group $(H,+,0)$, noted simply $H$, and a subgroup $Q$ of $H$, the {\em quotient} $\faktor{H}{Q}$ is the group $(\setcond{ \overline{x} }{ x \in H},+,\overline{0})$ where $\overline{x} = \setcond{y\in H}{y = x +q, q\in Q}$ is the {\em equivalence class} of $x$ and where the $+$ operation verifies $\overline{x} + \overline{y} = \overline{x+y}$.
If $G$ is a graph with vertex set a group $H$ and $Q$ is a subgroup of $H$, then the {\em quotient graph} $\faktor{G}{Q}$ over the vertices $\faktor{H}{Q}$ is defined by identifying the vertices in the same equivalence class. Similarly, if  $W = s_0, \dots, s_n$ is a walk on $G$, then the {\em quotient walk} $W'$ on  $\faktor{G}{Q}$ is the sequence $\overline{s_0}, \dots, \overline{s_n}$.


\medskip

Now, we count the number of signed complete graphs on four vertices. This result will be useful in the proof of Lemma~\ref{lem:ProdCyclePas4}.

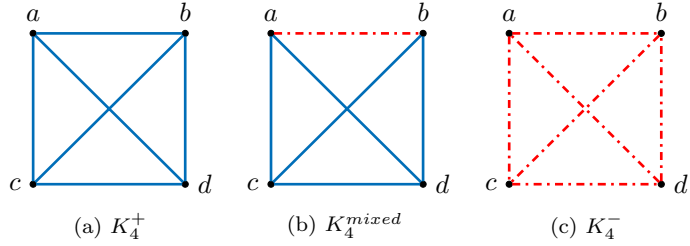
\begin{figure}[t]
\centering
\subfloat[$K_4^+$]{
\begin{tikzpicture}
\tikzstylemacro{}

\node[n, label=left:$c$] (C) at (0,-1) {};
\node[n, label=above:$a$] (A) at  (0,1) {};
\node[n, label=right:$d$] (D) at (2,-1) {};
\node[n, label=above:$b$] (B) at (2,1) {};
\draw[g] (A)--(B);
\draw[g] (D)--(A)--(C)--(D)--(B)--(C);
\end{tikzpicture}
}
\subfloat[$\KIVMixed$ \label{fig:K4Mixed}]{
\begin{tikzpicture}
\tikzstylemacro{}

\node[n, label=left:$c$] (C) at (0,-1) {};
\node[n, label=above:$a$] (A) at  (0,1) {};
\node[n, label=right:$d$] (D) at (2,-1) {};
\node[n, label=above:$b$] (B) at (2,1) {};
\draw[r] (A)--(B);
\draw[g] (D)--(A)--(C)--(D)--(B)--(C);
\end{tikzpicture}
}
\subfloat[$K_4^-$]{
\begin{tikzpicture}
\tikzstylemacro{}

\node[n, label=left:$c$] (C) at (0,-1) {};
\node[n, label=above:$a$] (A) at  (0,1) {};
\node[n, label=right:$d$] (D) at (2,-1) {};
\node[n, label=above:$b$] (B) at (2,1) {};
\draw[r] (A)--(B);
\draw[r] (D)--(A)--(C)--(D)--(B)--(C);
\end{tikzpicture}
}
\caption{The three complete signed graphs of order $4$}
\label{fig:3completeK4}
\end{figure}

\begin{theorem}\label{thm:SignedCompleteOrder4}
There are three complete signed graphs of order 4 (see Figure~\ref{fig:3completeK4}). They are the signed graph $K_4^+ = (K_4,\varnothing)$ with only positive edges, the signed graph $K_4^- = (K_4, E(K_4))$ with only negative edges and the signed graph $\KIVMixed = (K_4, \{ab\})$ where $a$ and $b$ are two vertices of $K_4$.
\end{theorem}

\begin{proof}
Let $(K_4,\sigma)$ be a complete signed graph on four vertices. Arbitrarily choose $u$ to be one of the vertices of $(K_4,\sigma)$. By  switching the neighbors of $u$ if needed, we can suppose that $u$ is only incident to positive edges. Let $x,y,z$ be the other three vertices of $(K_4,\sigma)$. If the triangle $xyz$ is all positive, then $(K_4,\sigma) = K_ 4^+$, if the triangle is all negative, then by switching $u$, we get $(K_4,\sigma) = K_ 4^-$. If the triangle has only one  negative edge, then $(K_4,\sigma) = K_4^{mixed}$. Otherwise, the triangle has two negative edges, by switching the vertex with the two negative edges, we get $(K_4,\sigma) = K_4^{mixed}$.
\end{proof}

\subsection{Beginning of the proof of Lemma~\ref{lem:ProdCyclePas4}}

Our goal is to prove Lemma~\ref{lem:ProdCyclePas4}. 
For that, take some integers $p$ and $q$, let $(P,\pi) =  UC_q\ssquare BC_{2p+1}$, and suppose that, by absurd, $\chi_s(P,\pi ) \leq 4$.

\begin{claim}
\label{claim:toK4Mixed}
We have $ (P,\pi) \sto \KIVMixed$.
\end{claim}
\begin{proof}

Since $\chi_s(P,\pi) \leq 4$, $(P,\pi) \sto (K_4,\rho)$ for some signature $\rho$ of $K_4$. 

Every equivalent signature of $BC_{2p+1}$ has at least one positive edge. Similarly, every equivalent signature of $UC_q$ has at least one negative edge.
Thus, in every equivalent signature of $(P,\pi)$, there is at least one positive edge and one negative edge. So $(H,\rho)$ cannot be $(K_4,\varnothing)$ nor $(K_4,E(K_4))$. 
By Theorem~\ref{thm:SignedCompleteOrder4}, since there are only three complete signed graphs of order $4$, $(H,\rho)$ is $\KIVMixed$.
\end{proof}

From now on, we suppose that we fixed a homomorphism $\varphi$ of $(P,\pi)$ to $\KIVMixed$. We label the vertices of $\KIVMixed$ as in Figure~\ref{fig:K4Mixed}. Therefore, there exists a signed graph $(P,\pi') \equiv (P,\pi)$ for which $v \mapsto \varphi(v)$ is a coloring.

The proof of Lemma~\ref{lem:ProdCyclePas4} is divided into four parts. First, by considering the graph $P$ as a toroidal grid, we define what we mean for a walk to make a ``turn'' around the torus in subsection~\ref{sec:turns}. 
Then, by considering the coloring of $(P,\pi')$ corresponding to $\varphi$ and the connected components of $(P,\pi')$ induced by colors $a$ and $b$, we link the number of ``crossings'' of some boundaries of the components with a vertical (or horizontal) cycle and the number of $ab$ edges of this cycle in subsection~\ref{sec:boundary}. In subsection~\ref{sec:crossings-turns}, we connect this number of `crossings'' to the number of turns and we conclude the proof in subsection~\ref{sec:proof-ending}.


\subsection{Number of turns in $P$}
\label{sec:turns}

The goal of this subsection is twofold. 
First, we want to establish another definition of $P$ as a toroidal grid \ie the quotient of some infinite grid. Secondly, we want to define the quantities $\tau_x(W)$ and $\tau_y(W)$ for each closed walk $W$ of $T$. They represent the number of turns in each direction of the torus made by the closed walk $W$.

\begin{figure}[!t]
\centering
\begin{tikzpicture}[scale=0.8]
\tikzstylemacro

\foreach \i in {0,1,2,3}{
	\foreach \j in {0,1}{
		\foreach \y in {0,1,2,3}{
			\foreach \x in {0,1,2}{
				\pgfmathtruncatemacro\result{\x  + 3 * \y}
				\pgfmathtruncatemacro\k{\x + 3 * \i}
				\pgfmathtruncatemacro\l{\y + 4 * \j}
				\node[v] (x\k y\l) at (7/8*\k,7/8*\l) {\result};
			}	
		}
		
		\coordinate (ctl\i\j) at (7/8*3*\i - 7/16,7/8*4*\j + 7/16 +7/8*3) {};
		\coordinate (ctr\i\j) at (7/8*3*\i + 7/16 + 7/8*2,7/8*4*\j + 7/16+7/8*3) {};
		\coordinate (cbl\i\j) at (7/8*3*\i - 7/16,7/8*4*\j - 7/16) {};
		\coordinate (cbr\i\j) at (7/8*3*\i + 7/16 + 7/8*2,7/8*4*\j - 7/16) {};
		
		\draw[b] (ctl\i\j) -- (ctr\i\j) -- (cbr\i\j) -- (cbl\i\j) -- (ctl\i\j);
	}
}

\foreach \k in {-1,12}{
	\foreach \l in {0,...,7}{
		\coordinate (x\k y\l) at (7/8*\k,7/8*\l) {};
	}
}
\foreach \l in {-1,8}{
	\foreach \k in {0,...,11}{
		\coordinate (x\k y\l) at (7/8*\k,7/8*\l) {};
	}
}

\foreach \i in {0,1,2,3}{
	\foreach \j in {0,1}{
		\foreach \y in {0,1,2,3}{
			\foreach \x in {0,1,2}{
				\pgfmathtruncatemacro\k{\x + 3 * \i}
				\pgfmathtruncatemacro\l{\y + 4 * \j}
				\pgfmathtruncatemacro\kplusun{\k + 1}
				\pgfmathtruncatemacro\kmoinsun{\k - 1}
				\pgfmathtruncatemacro\lplusun{\l +1}
				\pgfmathtruncatemacro\lmoinsun{\l -1}
				\draw (x\k y\l) -- (x\kplusun y\l);
				\draw (x\k y\l) -- (x\kmoinsun y\l);
				\draw (x\k y\l) -- (x\k y\lplusun);
				\draw (x\k y\l) -- (x\k y\lmoinsun);
			}	
		}
	}
}

\end{tikzpicture}
\caption{A subgraph of the graph $G$. Vertices with the same label are identified in $P$. Here $q=4$ and $2p+1=3$.}
\label{fig:Z2}
\end{figure}
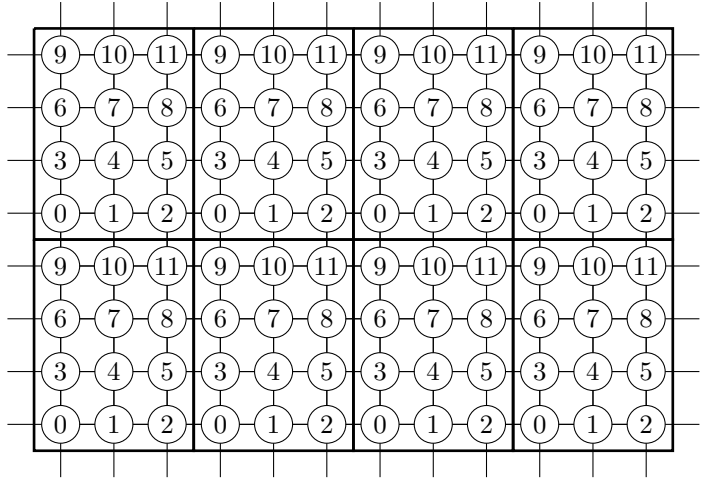

\begin{definition}
We can associate with $\ZZ^2$ an infinite graph ${G^\infty}$ whose vertex set $V({G^\infty})$ is the set $\setcond{v_{x,y}}{(x,y) \in \ZZ^2}$ and whose edge set is the set of pairs $\{v_{x,y}v_{x',y'}\}$, where either $x = x'$ and $\abs{y -y'} = 1$, or  $y = y'$ and $\abs{x -x'} = 1$. 
We can then redefine the graph $P$ as the quotient $\faktor{{G^\infty}}{Q}$ where $Q = \ZZ_{2p+1} \times \ZZ_{q}$. In other words take the graph ${G^\infty}$ where we identify each vertex $v_{x,y}$ with $v_{x',y'}$ when $x-x'$ is a multiple of  $2p+1$ and $y-y'$ is a multiple of $q$. The graph ${G^\infty}$ can be seen as an unfolding of the toroidal grid $P$. Figure~\ref{fig:Z2} represents a subgraph of ${G^\infty}$ when $q=4$ and $2p+1= 3$.
 An edge of ${G^\infty}$ of the form $v_{u,w}v_{u+i,w}$ (\resp $v_{u,w}v_{u,w+i}$) for $i\in \set{-1,1}$, is an \emph{horizontal} (\resp \emph{vertical}) edge of ${G^\infty}$. An edge $e$ of $P$ is an \emph{horizontal} (\resp \emph{vertical}) edge if it is the quotient of horizontal (resp. vertical) edges of ${G^\infty}$.
\end{definition}

\begin{definition}
Let $W_{G^\infty}$ be a walk in ${G^\infty}$ and $W_P$ a walk in $P$. We say that $W_{G^\infty}$ is a {\em representation} of $W_P$ if and only if $\faktor{W_{G^\infty}}{Q} = W_P$. We also say that $W_{G^\infty}$ {\em represents} $W_P$.
\end{definition}

By definition, all representations of $W_P$ have the same number of vertices as $W_P$. Let us make the following observation on the representations of a walk $W_P$. 



\begin{obs}
\label{claim:unicite-representant}
If $W_{G^\infty}^1 = (s^1_i)_{0\leq i \leq n}$ and $W_{G^\infty}^2= (s^2_i)_{0\leq i \leq n}$ are two walks (of the same length) in ${G^\infty}$ representing $W_P$, then there exist $\alpha,\beta \in \ZZ$ such that for all $i \in \set{0,\dots,n}$, if $s^1_i = v_{x,y}$, then $s^2_i = v_{x + \alpha (2p+1), y + \beta q}$. In particular, if they have the same starting vertices, then $W_{G^\infty}^1 = W_{G^\infty}^2$.
\end{obs}

%
%

We are now ready to define what is a turn of a walk around the torus.

\begin{definition}
Let $W_{G^\infty}$ be a walk in ${G^\infty}$ starting with $v_{x,y}$ and ending with $v_{z,t}$. We define the {\em number of horizontal turns} $\tau_x$ and  the {\em number of vertical turns} $\tau_y$ of $W_{G^\infty}$ by:
$$ \tau_x(W_{G^\infty}) = \abs{\frac{z-x}{2p+1}}, \tau_y(W_{G^\infty}) = \abs{\frac{t-y}{q}}.$$

For  a closed walk $W_P$ in $P$, let $\tau_x(W_P) = \tau_x(W_{G^\infty})$ (resp. $\tau_y(W_P) =\tau_y(W_{G^\infty})$) be the number of horizontal (resp. vertical) turns of $W_p$ where $W_{G^\infty}$ is an arbitrary representation of $W_P$.
\end{definition}

\begin{claim}
The two quantities $\tau_x(W_P)$ and $\tau_y(W_P)$ are integers and do not depend on the choice of the representation $W_{G^\infty}$ of $W_P$.
\end{claim}

\begin{proof}
First if $W_P$ is a closed walk in $P$ and $W_{G^\infty}$ represents $W_P$, then $\overline{v_{x,y}} = \overline{v_{z,t}}$ thus $z = x + n(2p+1)$ and $t=y+mq$ for some integers $n,m\in \ZZ$. Hence $\tau_x(W_{G^\infty})$ and $\tau_y(W_{G^\infty})$ are integers.

Now take two representations $W_{G^\infty}^1$ and $W_{G^\infty}^2$ of $W_P$. By Observation~\ref{claim:unicite-representant}, if $W_{G^\infty}^1$  starts at $v_{x_1,y_1}$ and ends at $v_{z_1,t_1}$ while $W_{G^\infty}^2$ starts at $v_{x_2,y_2}$ and ends at $v_{z_2,t_2}$, then $x_2 = x_1 + \alpha (2p+1)$, $y_2 = y_1 + \beta q$, $z_2 = z_1 + \alpha (2p+1)$ and $t_2 = t_1 + \beta q$. Thus $\tau_x(W_{G^\infty}^1) = \tau_x(W^2_{G^\infty})$ and $\tau_y(W_{G^\infty}^1) = \tau_y(W^2_{G^\infty})$. Hence this quantity is well defined for $W_P$.
\end{proof}

The main result of this subsection is the following proposition.

\begin{proposition}
\label{prop:sum-tau}
If $W_P$ is a closed walk in $P$ of even length, then:

$$  q \tau_y(W_P)  +\tau_x(W_P) \equiv 0 \pmod 2.$$
\end{proposition}

\begin{proof}
Let $W_{G^\infty}$ be a representation of $W_P$ in ${G^\infty}$ starting at $v_{x,y}$ and ending at $v_{z,t}$. For each horizontal (\resp vertical) edge $e$ of the form $v_{u,w}v_{u+i,w}$ (\resp $v_{u,w}v_{u,w+i}$) for $i\in \set{-1,1}$, let $\ell(e) = i$. Let $E_h(W_{G^\infty})$ be the set of horizontal edges of $W_{G^\infty}$ and $E_v(W_{G^\infty})$ the set of vertical edges of $W_{G^\infty}$.  We then have:

$$
\sum_{e \in E_h(W_{G^\infty})} \ell(e) \equiv z-x \equiv (2p+1)\tau_x(W_P) 
 \equiv \tau_x(W_P) \pmod 2,
$$
and 
$$
 \sum_{e \in E_h(W_{G^\infty})} \ell(e) \equiv \sum_{e \in E_h(W_{G^\infty})} 1 \equiv \abs{E_h(W_{G^\infty})} \pmod 2.
$$
Similarly,
$$
\abs{E_v(W_{G^\infty})} \equiv t-y
 \equiv q\tau_y(W_P) \pmod 2.
$$
As $W_P$ and $W_{G^\infty}$ are of even length, we get:

$$0 \equiv \abs{E(W_{G^\infty})} \equiv q \tau_y(W_P)  + \tau_x(W_P) \pmod 2.$$
\end{proof}

\subsection{Regions induced by a coloring of $(P,\pi)$}
\label{sec:boundary}

The aim of this section is to define a suitable set of walks in order to apply Proposition~\ref{prop:sum-tau}. For this, we will introduce several topological notions.

\begin{definition}
Let $P_{AB} = P[\varphi^{-1}\{a,b\}]$ and $P_{CD} = P[\varphi^{-1}\{c,d\}]$, the subgraphs of $P$  induced by the vertices colored $a$ and $b$ and by the vertices colored $c$ and $d$, respectively.
A {\em region} $X$ of $P$ is a connected component of $P_{AB}$ or $P_{CD}$. We say that $X$ is {\em of type  $ab$} in the first case and {\em of type $cd$} in the latter.
{\em The boundary} $\partial X$ of a region $X$ is the subset of vertices of $X$ that are adjacent to a vertex not in $X$:
$$ \partial X = \setcond{x \in X}{ N(x) \nsubseteq X}.$$


%
%
\end{definition}

\begin{claim}
\label{claim:flat-border}
The configuration of Figure~\ref{fig:forbidden-flat-border} cannot appear in the coloring of $(P,\pi')$. That is to say, for a region $X$ there do not exist two vertices $x,y \in X$ and $w,z \notin X$ such that $xy$, $xz$, $yw$ and $wz$ belong to $E(P)$. We call this configuration the {\em flat border configuration}.
\end{claim}

\begin{proof}
Suppose to the contrary that the configuration appears. Then, the cycle $xywz$ of length $4$ is unbalanced. Thus, before switching, $xywz$ was already an unbalanced cycle of length $4$ in  $(P,\pi)$ since balance is preserved by switching. By definition of $(P,\pi)$ as a Cartesian product of cycles, the signs of $xy$ and $wz$ are the same. It is also the case for $zx$ and $wy$. Thus this cycle is balanced (it has an even number of negative edges), a contradiction.
\end{proof}

\begin{figure}[t]
\centering
\begin{tikzpicture}
\tikzstylemacro{}
\node[v] (x) at (1.5,1.5) {$x$};
\node[v] (y) at (0,0) {$w$};
\node[v] (a) at (1.5,0) {$y$};
\node[v] (b) at (0,1.5) {$z$};
\draw[b] (x) -- (a) -- (y) -- (b) -- (x);
\draw[b,dashed] (0.75,2.5) -- (0.75,-.5);
\coordinate[label=right:$\in X$] () at (0.75,2) {};
\coordinate[label=left:$\notin X$] () at (0.75,2) {};
\end{tikzpicture}
\caption{The flat border configuration when $x,y\in X$ and $z, w \notin X$, where $X$ is a region.}
\label{fig:forbidden-flat-border}
\end{figure}
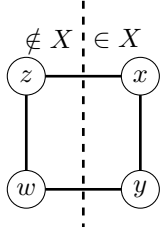

\begin{figure}[t]
\centering

\begin{tikzpicture}
\tikzstylemacro{}
\node[v] (x) at (1.5,1.5) {$x$};
\node[v] (y) at (0,0) {$y$};
\node[v] (a) at (1.5,0) {$w$};
\node[v] (b) at (0,1.5) {$z$};
\draw[b] (x) -- (a) -- (y) -- (b) -- (x);
\draw[b,dashed] (0.75,2.5) -- (0.75,0.75) -- (-0.75,0.75);
\coordinate[label=right:$\in X$] () at (0.75,2) {};
\coordinate[label=left:$\notin X$] () at (0.75,2) {};
\end{tikzpicture}

\caption{The vertices $x$ and $y$ of the region $X$ are border neighbors.}
\label{fig:borderNeighbors}
\end{figure}
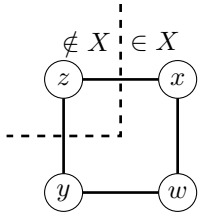

\begin{definition}
Two vertices $x$ and $y$ on the boundary of the region $X$ are {\em border neighbors} if $x$ and $y$ have a common neighbor in $X$ and a common neighbor in $P\setminus X$ (see Figure~\ref{fig:borderNeighbors}). 
We note $BN(x)$ the set of border neighbors of $x$.

{\em A border} $B$ of a region $X$ is a subset of $\partial X$ corresponding to an equivalence class for the transitive closure of the border neighborhood relation (see Figure~\ref{fig:region-border-def}). That is to say, two vertices $x$ and $y$ of $\partial X$ are in the same border $B$ of $X$ if and only if there exists a sequence $u_0,u_1,\dots,u_k$ of vertices of $B$ such that $u_0 = x$, $u_k = y$ and for all $0 \leq i < k$, $u_i$ and $u_{i+1}$ are border neighbors. 
\end{definition}

\begin{figure}[!t]
\centering
\begin{tikzpicture}
\tikzstylemacro{}
\node[v] (x00) at (0,0) {$c$};
\node[v,rectangle] (x10) at (1,0) {$a$};
\node[v] (x20) at (2,0) {$c$};
\node[v,rectangle] (x30) at (3,0) {$a$};
\node[v] (x40) at (4,0) {$c$};
\node[v] (x50) at (5,0) {$d$};

\node[v,rectangle] (x01) at (0,1) {$a$};
\node[v] (x11) at (1,1) {$b$};
\node[v,rectangle] (x21) at (2,1) {$a$};
\node[v] (x31) at (3,1) {$b$};
\node[v,rectangle] (x41) at (4,1) {$a$};
\node[v] (x51) at (5,1) {$c$};

\node[v] (x02) at (0,2) {$d$};
\node[v,rectangle] (x12) at (1,2) {$a$};
\node[v] (x22) at (2,2) {$b$};
\node[v] (x32) at (3,2) {$a$};
\node[v] (x42) at (4,2) {$b$};
\node[v,rectangle] (x52) at (5,2) {$a$};

\node[v] (x03) at (0,3) {$b$};
\node[v] (x13) at (1,3) {$c$};
\node[v,rectangle] (x23) at (2,3) {$a$};
\node[v] (x33) at (3,3) {$b$};
\node[v,rectangle] (x43) at (4,3) {$a$};
\node[v] (x53) at (5,3) {$c$};

\node[v] (x04) at (0,4) {$a$};
\node[v] (x14) at (1,4) {$b$};
\node[v] (x24) at (2,4) {$d$};
\node[v,rectangle] (x34) at (3,4) {$a$};
\node[v] (x44) at (4,4) {$c$};
\node[v] (x54) at (5,4) {$d$};

\foreach \i in {0,1,2,3,4}
	\draw (-0.6,\i) -- (x0\i) -- (x1\i) -- (x2\i) -- (x3\i) -- (x4\i) -- (x5\i) -- (5.6,\i);
	
\foreach \i in {0,1,2,3,4,5}
	\draw (\i, -0.6) -- (x\i0) -- (x\i1) -- (x\i2) -- (x\i3) -- (x\i4) -- (\i,4.6);
	
\draw[line width =2pt] (0.5,0.5) -- (0.5,-0.5) -- (1.5,-0.5) -- (1.5,0.5) -- (2.5,0.5) -- (2.5, -0.5) -- (3.5,-0.5) -- (3.5,0.5) -- (4.5,0.5) -- (4.5,1.5) -- (5.5,1.5) -- (5.5,2.5) -- (4.5,2.5) -- (4.5,3.5) -- (3.5,3.5) -- (3.5,4.5) -- (2.5,4.5) -- (2.5,3.5) -- (1.5,3.5) -- (1.5,2.5) -- (0.5,2.5) -- (0.5,1.5) -- (-0.5,1.5) -- (-0.5,0.5) -- (0.5,0.5);

\begin{scope}[very thick,decoration={
    markings,
    mark=at position 0.5 with {\arrow{>}}}
    ] 
    \draw[line width=2pt, dashdotted, postaction={decorate}] (0.85,0.85) -- (0.15,1);
    \draw[line width=2pt, dashdotted, postaction={decorate}] (0.15,1) -- (0.85,1.15) -- (x12) -- (x22) -- (x23) -- (2.85,3.15) -- (x34) -- (3.15,3.15) -- (x43) -- (4.15,2.15) -- (x52) -- (4.15,1.85) -- (x41) -- (3.15,0.85) -- (x30) -- (2.85,0.85) -- (x21) -- (1.15,0.85) -- (x10) -- (0.85,0.85);
\end{scope}

\end{tikzpicture}
\caption{A region $X$ delimited by the bold line and the only border $B$ of $X$ is represented by the square vertices. The dotted line represents the only walk in $\mathcal W_B$.}
\label{fig:region-border-def}
\end{figure}
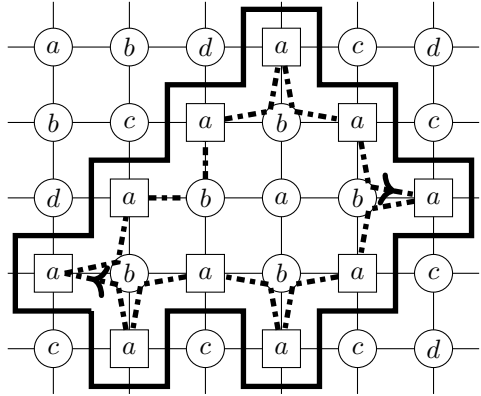

\begin{claim}
\label{claim:same-color-border}
All vertices of a border $B$ of a region $X$ have the same color called the {\em color of} $B$.
\end{claim}

\begin{proof}
By definition of $B$ it suffices to show that any two border neighbors $x$ and $y$ have the same color. Let $z$ be their common neighbor in $X$. Without loss of generality, suppose $X$ is of type $ab$ and $z$ has color $b$. 
Since the coloring is proper, $x$ and $y$ have color $a$.
\end{proof}


\begin{claim}
\label{claim:number-of-border-neighbors}
A vertex $x$ of a border $B$ has an even number of border neighbors. Moreover if $BN(x) = \varnothing$, then $X = \set{x}$.
\end{claim}

\begin{figure}[!t]
\centering
\subfloat[Case 1: $\abs{BN(x)} = 1$.\label{fig:even-number-border-neighbor-1}]{
\begin{tikzpicture}
\tikzstylemacro

\node[v] (x) at (0,0) {$x$};
\node[v, rectangle] (w) at (0,1) {$i$};

\node[v] (1) at (1,1) {$1$};
\node[v, rectangle] (2) at (1,-1) {$2$};
\node[v, rectangle] (3) at (-1,-1) {$3$};

\node[v] (i1) at (1,0) {$j$};
\node[v] (i2) at (0,-1) {$k$};
\node[v, rectangle] (i3) at (-1,0) {$\ell$};

\node[v, rectangle] (i4) at (-1,1) {$4$};

\draw[b] (i3) -- (x) -- (w) -- (1);
\draw[b] (i1) -- (2) -- (i2) -- (3) -- (i3) -- (i4) -- (w);
\draw[b] (i2) -- (x);
\draw[b] (x) -- (i1) -- (1);
\end{tikzpicture}
}
\hskip 1cm
\subfloat[Case 2: $\abs{BN(x)} = 3$.\label{fig:even-number-border-neighbor-2}]{
\begin{tikzpicture}
\tikzstylemacro

\node[v] (x) at (0,0) {$x$};
\node[v, rectangle] (w) at (0,1) {$i$};

\node[v] (1) at (1,1) {$1$};
\node[v] (2) at (1,-1) {$2$};
\node[v] (3) at (-1,-1) {$3$};

\node[v] (i1) at (1,0) {$j$};
\node[v, rectangle] (i2) at (0,-1) {$k$};
\node[v] (i3) at (-1,0) {$\ell$};

\node[v, rectangle] (i4) at (-1,1) {$4$};

\draw[b] (i3) -- (x) -- (w) -- (1);
\draw[b] (i1) -- (2) -- (i2) -- (3) -- (i3) -- (i4) -- (w);
\draw[b] (i2) -- (x);
\draw[b] (x) -- (i1) -- (1);
\end{tikzpicture}
}
\hskip 1cm
\subfloat[Case: $\abs{BN(x)} = 0$.\label{fig:even-number-border-neighbor-0}]{
\begin{tikzpicture}
\tikzstylemacro

\node[v] (x) at (0,0) {$x$};
\node[v, rectangle] (w) at (0,1) {$i$};

\node[v, rectangle] (1) at (1,1) {$1$};
\node[v, rectangle] (2) at (1,-1) {$2$};
\node[v, rectangle] (3) at (-1,-1) {$3$};

\node[v, regular polygon,regular polygon sides=3] (i1) at (1,0) {$j$};
\node[v, regular polygon,regular polygon sides=3] (i2) at (0,-1) {$k$};
\node[v, regular polygon,regular polygon sides=3] (i3) at (-1,0) {$\ell$};

\node[v, rectangle] (i4) at (-1,1) {$4$};

\draw[b] (i3) -- (x) -- (w) -- (1);
\draw[b] (i1) -- (2) -- (i2) -- (3) -- (i3) -- (i4) -- (w);
\draw[b] (i2) -- (x);
\draw[b] (x) -- (i1) -- (1);
\end{tikzpicture}
}
\caption{The two cases up to symmetry where $\abs{BN(x)}$ has an odd number of vertices and the case where $BN(x) = \emptyset$. The square vertices represent vertices not in the region of $x$, the circular ones are in the region while the triangular ones are undecided.}
\label{fig:even-number-border-neighbor}
\end{figure}
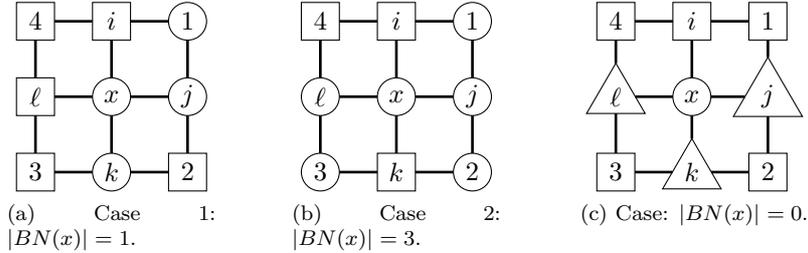

\begin{proof}
If $\abs{BN(x)}$ is odd, then we are in one of the first two cases of Figure~\ref{fig:even-number-border-neighbor}. We will use the notation of the figure.

If $\abs{BN(x)} = 1$, then up to rotation and symmetry, we can suppose that the vertex $1$ is the border neighbor of $x$ and that $j$ is their common neighbor in $X$. Thus $i \notin X$. Now $\ell \notin X$, as otherwise the vertices $i$, $\ell$, $4$ and $x$ would be in the flat border configuration, which cannot be by Claim~\ref{claim:flat-border}. The same argument implies $k \in X$ by considering $x$, $k$, $j$ and $2$. Thus $x$, $\ell$, $k$ and $3$ are in the flat border configuration. A contradiction.

If $\abs{BN(x)} = 3$, then up to rotation and symmetry, we can suppose that the vertex $4$ is not a border neighbor of $x$.
As $2$ is a border neighbor of $x$, one of $k$ and $j$ is in $X$ and the other is not. Without loss of generality, suppose $k \notin X$ and $j \in X$. As $3$ is a border neighbor of $x$, we have $\ell \in X$.  As $1$ is a border neighbor of $x$, we have $i \notin X$. Thus $4$, $i$, $\ell$ and $x$ are in the flat border configuration, a contradiction.

Now if $BN(x) = \varnothing$, we can suppose that $i \notin X$ as $x$ is in $\partial X$. Now to avoid the flat border configuration, $j$, $k$ and $\ell$ must not be in $X$. This proves that $X = \set{x}$. 
\end{proof}

We can now define the set of walks associated with the border.

\begin{definition}
\label{def:border-line}
We associate with a border $B$ of $X$, a set of closed walks $\mathcal{W}_B$ in $(P,\pi')$ included in $X$ (see Figure~\ref{fig:region-border-def}). This set of walks delimits the border of $X$.
We use $v_{i,j}$ to refer to the vertex $\overline{v_{i,j}}$ of $P$ for concision.

We will define the walks piece by piece.
In the particular case that $X$ has only one vertex, then $\mathcal{W}_B = \varnothing$. Now we can suppose that for each $x \in B$, we  have $BN(x) \neq \varnothing$ by Claim~\ref{claim:number-of-border-neighbors}.

\begin{figure}[!t]
\centering
\subfloat[The order of the vertices to choose.\label{fig:start-border-gen}]{
\begin{tikzpicture}
\tikzstylemacro

\node[v] (x) at (0,0) {$x$};
\node[v, rectangle] (w) at (0,1) {$w$};

\node[v, regular polygon,regular polygon sides=3] (1) at (1,1) {$1$};
\node[v, regular polygon,regular polygon sides=3] (2) at (1,-1) {$2$};
\node[v, regular polygon,regular polygon sides=3] (3) at (-1,-1) {$3$};

\node[v, regular polygon,regular polygon sides=3] (i1) at (1,0) {};
\node[v, regular polygon,regular polygon sides=3] (i2) at (0,-1) {};
\node[v, regular polygon,regular polygon sides=3] (i3) at (-1,0) {};

\node[v, regular polygon,regular polygon sides=3] (i4) at (-1,1) {};

\draw[b,dashed] (i3) -- (x) -- (w) -- (1) -- (i1) -- (2) -- (i2) -- (3) -- (i3) -- (i4) -- (w);
\draw[b,dashed] (i2) -- (x) -- (i1);
\end{tikzpicture}
}
\hskip 1cm
\subfloat[Case 1.\label{fig:start-border-1}]{
\begin{tikzpicture}
\tikzstylemacro

\node[v] (x) at (0,0) {$x$};
\node[v, rectangle] (w) at (0,1) {$w$};

\node[v] (1) at (1,1) {$y$};
\node[v, regular polygon,regular polygon sides=3] (2) at (1,-1) {$2$};
\node[v, regular polygon,regular polygon sides=3] (3) at (-1,-1) {$3$};

\node[v] (i1) at (1,0) {$z$};
\node[v, regular polygon,regular polygon sides=3] (i2) at (0,-1) {};
\node[v, regular polygon,regular polygon sides=3] (i3) at (-1,0) {};

\node[v, regular polygon,regular polygon sides=3] (i4) at (-1,1) {};

\draw[b,dashed] (i3) -- (x) -- (w) -- (1);
\draw[b,dashed] (i1) -- (2) -- (i2) -- (3) -- (i3) -- (i4) -- (w);
\draw[b,dashed] (i2) -- (x);
\draw[b, line width=2pt] (x) -- (i1) -- (1);
\end{tikzpicture}
}
\hskip 1cm
\subfloat[Case 2.\label{fig:start-border-2}]{
\begin{tikzpicture}
\tikzstylemacro

\node[v] (x) at (0,0) {$x$};
\node[v, rectangle] (w) at (0,1) {$w$};

\node[v, rectangle] (1) at (1,1) {$1$};
\node[v] (2) at (1,-1) {$y$};
\node[v, regular polygon,regular polygon sides=3] (3) at (-1,-1) {$3$};

\node[v, rectangle] (i1) at (1,0) {};
\node[v] (i2) at (0,-1) {$z$};
\node[v, regular polygon,regular polygon sides=3] (i3) at (-1,0) {};

\node[v, regular polygon,regular polygon sides=3] (i4) at (-1,1) {};

\draw[b,dashed] (i3) -- (x) -- (w) -- (1) -- (i1) -- (2);
\draw[b,dashed] (i2) -- (3) -- (i3) -- (i4) -- (w);
\draw[b,dashed] (x) -- (i1);

\draw[b, line width=2pt] (x) -- (i2) -- (2);
\end{tikzpicture}
}
\hskip 1cm
\subfloat[Case 3.\label{fig:start-border-3}]{
\begin{tikzpicture}
\tikzstylemacro

\node[v] (x) at (0,0) {$x$};
\node[v, rectangle] (w) at (0,1) {$w$};

\node[v, rectangle] (1) at (1,1) {$1$};
\node[v, rectangle] (2) at (1,-1) {$2$};
\node[v] (3) at (-1,-1) {$y$};

\node[v, rectangle] (i1) at (1,0) {};
\node[v, rectangle] (i2) at (0,-1) {};
\node[v] (i3) at (-1,0) {$z$};

\node[v, regular polygon,regular polygon sides=3] (i4) at (-1,1) {};

\draw[b,dashed] (i3) -- (x) -- (w) -- (1) -- (i1) -- (2) -- (i2);
\draw[b,dashed] (i2) -- (3) -- (i3) -- (i4) -- (w);
\draw[b,dashed] (i2) -- (x) -- (i1);

\draw[b, line width=2pt] (x) -- (i3) -- (3);
\end{tikzpicture}
}
\caption{The first step in constructing the walk. 
The square vertices represent vertices not in the region of $x$, the circular ones are in the region while the triangular ones are undecided. The dashed edges are the edges of ${G^\infty}$ while the bold edges are the edges of the walk.}
\label{fig:start-border}
\end{figure}
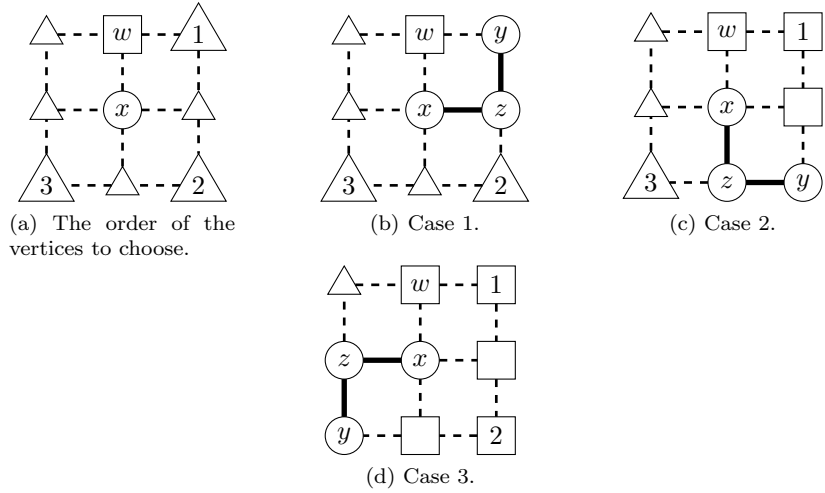

First pick an arbitrary vertex $x$ of $B$. The vertex $x$ is a border vertex thus there exists at least one vertex $w$ adjacent to $x$ which is not in $X$. In case there are more than one such vertex, we choose one of them arbitrarily. Up to rotation of the coordinate system, we can suppose $x = v_{i,j}$ and $w = v_{i,j+1}$. We will choose $y \in BN(x)$ according to the order in Figure~\ref{fig:start-border-gen}. Meaning the first vertex among $v_{i+1,j+1}$, $v_{i+1,j-1}$ and $v_{i-1,j-1}$ that belongs to $BN(x)$. The three cases are depicted in Figure~\ref{fig:start-border-1},~\ref{fig:start-border-2} and~\ref{fig:start-border-3}. Note that as $BN(x)$ is non-empty, $BN(x)$ has at least two vertices by Claim~\ref{claim:number-of-border-neighbors}, thus we are in at least one of the three cases above. Through the construction, the ``turn left'' property, which implies that the vectors $\overrightarrow{s_{2i}s_{2i+1}}$ and $\overrightarrow{s_{2i+1}s_{2i+2}}$ form a direct base, will be conserved.

Now that we have $x$ and $y$ we can start to construct our walk $W$, by taking $s_0 = x$, $s_1 = z$ and $s_2 = y$  where $z$ is the common neighbor of $x$ and $y$ which is in $X$.

\begin{figure}[!t]
\centering
\subfloat[The order of the vertices to choose.\label{fig:middle-border-gen}]{
\begin{tikzpicture}
\tikzstylemacro

\node[v] (x) at (0,0) {$x$};
\node[v, rectangle] (w) at (0,1) {$w$};

\node[v, regular polygon,regular polygon sides=3] (1) at (1,1) {$1$};
\node[v, regular polygon,regular polygon sides=3] (2) at (1,-1) {$2$};
\node[v, regular polygon,regular polygon sides=3] (3) at (-1,-1) {$3$};

\node[v, regular polygon,regular polygon sides=3] (i1) at (1,0) {};
\node[v, regular polygon,regular polygon sides=3] (i2) at (0,-1) {};
\node[v] (i3) at (-1,0) {$v$};

\node[v] (i4) at (-1,1) {$u$};

\draw[b,dashed] (x) -- (w) -- (1) -- (i1) -- (2) -- (i2) -- (3) -- (i3) ;
\draw[b,dashed] (i4) -- (w);
\draw[b,dashed] (i2) -- (x) -- (i1);
\draw[b, line width=2pt] (i4) -- (i3) -- (x);
\end{tikzpicture}
}
\hskip 1cm
\subfloat[Case 1.\label{fig:middle-border-1}]{
\begin{tikzpicture}
\tikzstylemacro

\node[v] (x) at (0,0) {$x$};
\node[v, rectangle] (w) at (0,1) {$w$};

\node[v] (1) at (1,1) {$y$};
\node[v, regular polygon,regular polygon sides=3] (2) at (1,-1) {$2$};
\node[v, regular polygon,regular polygon sides=3] (3) at (-1,-1) {$3$};

\node[v] (i1) at (1,0) {$z$};
\node[v, regular polygon,regular polygon sides=3] (i2) at (0,-1) {};
\node[v] (i3) at (-1,0) {$v$};

\node[v] (i4) at (-1,1) {$u$};

\draw[b,dashed] (x) -- (w) -- (1);
\draw[b,dashed] (i1) -- (2) -- (i2) -- (3) -- (i3);
\draw[b,dashed] (i4) -- (w);
\draw[b,dashed] (i2) -- (x);
\draw[b, line width=2pt] (x) -- (i1) -- (1);
\draw[b, line width=2pt] (i4) -- (i3) -- (x);
\end{tikzpicture}
}
\hskip 1cm
\subfloat[Case 2.\label{fig:middle-border-2}]{
\begin{tikzpicture}
\tikzstylemacro

\node[v] (x) at (0,0) {$x$};
\node[v, rectangle] (w) at (0,1) {$w$};

\node[v, rectangle] (1) at (1,1) {$1$};
\node[v] (2) at (1,-1) {$y$};
\node[v, regular polygon,regular polygon sides=3] (3) at (-1,-1) {$3$};

\node[v, rectangle] (i1) at (1,0) {};
\node[v] (i2) at (0,-1) {$z$};
\node[v] (i3) at (-1,0) {$v$};

\node[v] (i4) at (-1,1) {$u$};

\draw[b,dashed] (i3) -- (x) -- (w) -- (1) -- (i1) -- (2);
\draw[b,dashed] (i2) -- (3) -- (i3) -- (i4) -- (w);
\draw[b,dashed] (x) -- (i1);

\draw[b, line width=2pt] (x) -- (i2) -- (2);
\draw[b, line width=2pt] (i4) -- (i3) -- (x);
\end{tikzpicture}
}
\hskip 1cm
\subfloat[Case 3.\label{fig:middle-border-3}]{
\begin{tikzpicture}
\tikzstylemacro

\node[v] (x) at (0,0) {$x$};
\node[v, rectangle] (w) at (0,1) {$w$};

\node[v, rectangle] (1) at (1,1) {$1$};
\node[v, rectangle] (2) at (1,-1) {$2$};
\node[v] (3) at (-1,-1) {$y$};

\node[v, rectangle] (i1) at (1,0) {};
\node[v, rectangle] (i2) at (0,-1) {};
\node[v] (i3) at (-1,0) {$v,z$};

\node[v] (i4) at (-1,1) {$u$};

\draw[b,dashed]  (x) -- (w) -- (1) -- (i1) -- (2) -- (i2);
\draw[b,dashed] (i2) -- (3) -- (i3) -- (i4) -- (w);
\draw[b,dashed] (i2) -- (x) -- (i1);

\draw[b, line width=2pt] (i4) -- (i3) -- (3);
\draw (x) edge[b, line width=2pt, bend left] (i3) edge[b, line width=2pt, bend right] (i3) ;
\end{tikzpicture}
}
\caption{The next step in constructing the walk. We use the same notation as in Figure~\ref{fig:start-border}.
We constructed $s_{l-2} = u$, $s_{l-1} = v$, $s_l =x$. We then construct $s_{l+1} = z$ and $s_{l+2} = y$.}
\label{fig:middle-border}
\end{figure}
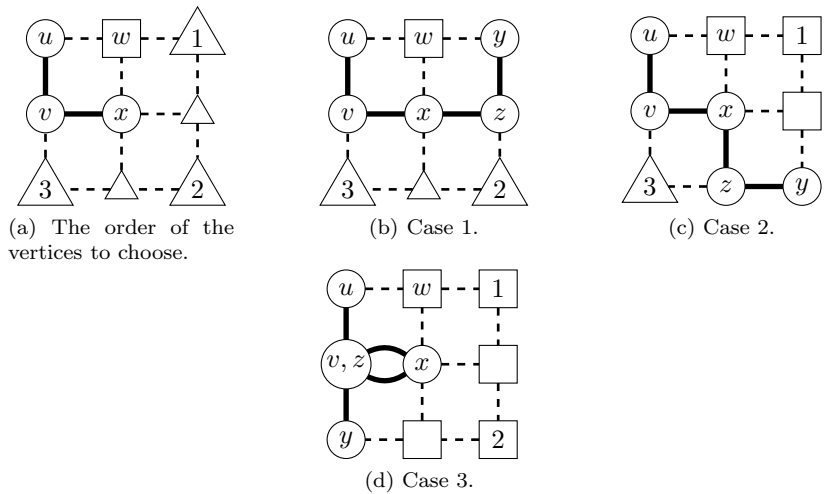

Suppose now that we have constructed the walk up to $s_0,\dots,s_{\ell-2},s_{\ell-1},s_{\ell}$ with $\ell$ even. If $s_{\ell-2} = s_0$, $s_{\ell-1} = s_1$ and $s_{\ell} = s_2$, then we stop and close this walk by removing the last two vertices. Otherwise we will construct $s_{\ell+1}$ and $s_{\ell+2}$. Suppose that $s_l = v_{i,j}$. Up to rotation of the coordinate system, we can suppose that $s_{\ell-2} = v_{i-1,j+1}$ and $s_{\ell-1} = v_{i-1,j}$. The vertex $s_{\ell-1}$ could in principle be $v_{i,j+1}$ but this would contradict  the ``turn left'' property. 
We construct $s_{\ell+2}$ as the first vertex among $v_{i+1,j+1}$, $v_{i+1,j-1}$ and $v_{i-1,j-1}$ that belong to $BN(x)$ (see Figure~\ref{fig:middle-border-gen}). The three cases are depicted in Figure~\ref{fig:middle-border-1},~\ref{fig:middle-border-2} and~\ref{fig:middle-border-3}. As before, since $BN(x)$ is non empty and of even cardinality, we are in one of those three cases. As in the first step, the vertex $s_{\ell+1}$ is the common neighbor of $s_{\ell}$ and $s_{\ell+2}$ in $X$.

If we stop and there are pairs of border neighbors that are not in the same walk, we can start the process again with this pair of vertices as the first and third vertices of the walk. To keep the assumption of the construction true, we must carefully choose the start vertex among these two in such a way that the ``turn left'' property is conserved.
\end{definition}

\begin{claim}
\label{claim:properties-border}
The construction of Definition~\ref{def:border-line} has the following properties:
\begin{enumerate}
\item the construction terminates and all walks are closed, \label{claim:properties-border-1}
\item the walks are of even length, \label{claim:properties-border-2}
\item all vertices with even indices have the same color, \label{claim:properties-border-3}
\item all vertices with odd indices have the same color which is different from the color of the vertices with even indices, \label{claim:properties-border-4}
\item all vertices of the border $B$ are vertices of some walk with even index, \label{claim:properties-border-5}
\item the number of occurrences of a vertex $x$ of the border $B$ in all the walks of $\mathcal W_B$  is given by $\abs{BN(x)} / 2$. \label{claim:properties-border-6}
\end{enumerate}
\end{claim}

\begin{proof}
Suppose we do not terminate. As the number of possible edges is finite, the sequence we construct is ultimately periodic. Since $s_0$, $s_1$, $s_2$ do not appear consecutively in this order in the rest of the sequence, as we did not stop, the sequence is not periodic. Thus there exists a first moment at which there exist $i$ and $j$ such that $s_{i-2},s_{i-1},s_i,s_{i+1},s_{i+2}$ and $s_{j-2},s_{j-1},s_j,s_{j+1},s_{j+2}$ are subsequences of the sequence we constructed, verifying $s_{i-2} \neq s_{j-2}$, $s_{i-1} \neq s_{j-1}$, $s_{i} = s_{j}$, $s_{i+1} = s_{j+1}$ and $s_{i+2} = s_{j+2}$. 
Note that knowing $s_{i-2}$, $s_i$ and $s_{i+2}$ imposes the choices of $s_{i-1}$ and $s_{i+1}$ by the ``turn left'' property, this is why these two indices exist. 
Without loss of generality, up to rotating the grid, we can assume that $s_i = v_{x,y}$ and $s_{i+2} = v_{x+1,y+1}$. By reversing the construction, we can observe that $s_{i-2}$ is the first border neighbor of $s_i$ among $v_{x-1,y+1}$, $v_{x-1,y-1}$ and $v_{x+1,y-1}$ in this order. In this case, $s_{j+2}$ and $s_{i-2}$ are uniquely determined by construction and are equal, a contradiction. 
Now the process terminates, thus the walks are closed by definition of the terminating condition. This proves~\ref{claim:properties-border-1}.

Since the walks are included in $X$ which is bipartite, they have even length which proves~\ref{claim:properties-border-2}.
In a similar way, all vertices with even indices are on the border $B$ of $X$ thus they have the same color by Claim~\ref{claim:same-color-border}, thus~\ref{claim:properties-border-3} is true. Thus all vertices with odd indices have the other color of $X$ which proves~\ref{claim:properties-border-4}.

We already saw that the vertices of even indices are on $B$. Suppose that $x$ is not part of a walk. We removed the case $\abs{BN(x)} = \varnothing$ by not considering those $B$'s thus there exists $y \in \abs{BN(x)}$. Then $x$ and $y$ are not in the same walk, thus we create a new one with these two vertices, a contradiction. This proves~\ref{claim:properties-border-5}.

Similarly,  if the number of occurrences is strictly smaller than $\abs{BN(x)} / 2$, we would have restarted the process in $x$. Now suppose that this number is strictly greater  than $\abs{BN(x)} / 2$. Then there exists a pair of border neighbors $x$ and $y$ that belong to two walks (and there is a vertex $z$ in between them in those two walks) by the pigeon hole principle. Since the construction of the walks only use the position of three consecutive vertices to decide the next two ones, the two walks are identical after  passing through $xzy$. By construction,  we can choose the start of the walks arbitrarily among the vertices of even indices by shifting the indices, thus we can consider that the two walks start by $xzy$. Thus the two walks are identical which cannot be the case as we would not have restarted to create the second walk.
\end{proof}


We define the set of closed walks $\mathcal{W}_a$ as the union of all closed walks $\mathcal{W}_B$ where $B$ is a border with color $a$.

Take $C$ to be a vertical or horizontal cycle of $P$. For the sake of simplicity, we will take $C$ to be the vertical cycle $\setcond{\overline{v_{x,y}}}{x = x_0 + n \ell,\ n \in \NN}$ where $\ell = 2p+1$ (\ie a $UC_{2q}$-layer). All the following definitions can be stated in the other case by symmetry.

Let $W$ be a closed walk in $P$ (\resp a representation of a closed walk of $P$ in ${G^\infty}$). We define a {\em positive crossing  of $C$ by $W$ in $P$} (\resp ${G^\infty}$) as a sub-walk $t_0,t_1,\dots,t_{k-1},t_k$ of $W$ (possibly going through the end of $W$ and going back at the beginning) such that $t_0 = \overline{v_{x_0-1,y}}$ (\resp $t_0 = v_{x_0 -1 + n \ell ,y}$ for $n\in \NN$) for some $y$, $t_i \in C$ for $i \in \set{1,\dots,k-1}$ and $t_k =  \overline{v_{x_0+1,y'}}$ (\resp $t_j = v_{x_0 +1 + n \ell ,y'}$ for $n\in \NN$) for some $y'$. The set of positive crossings $Cross^+_P(W,C)$ (\resp $Cross^+_{G^\infty}(W,C)$) is the set of all positive crossings of $C$ by $W$ in $P$ (\resp ${G^\infty}$).

We can similarly define a {\em negative crossing of $C$ by $W$ in $P$} (\resp ${G^\infty}$) by a sub-walk $t_0,t_1,\dots,t_{k-1},t_k$ of $W$ (possibly going through the end of $W$ and going back at the beginning) such that $t_0 = \overline{v_{x_0+1,y}}$ (\resp $t_0 = v_{x_0 +1 + n \ell ,y}$ for $n\in \NN$) for some $y$, $t_i \in C$ for $i \in \set{1,\dots,k-1}$ and $t_k =  \overline{v_{x_0-1,y'}}$ (\resp $t_j = v_{x_0 -1 + n \ell ,y'}$ for $n\in \NN$) for some $y'$. We note the corresponding set $Cross^-_P(W,C)$ (\resp $Cross^-_{G^\infty}(W,C)$).

\begin{claim}
\label{claim:cross-quotient}
If $W_{G^\infty}$ represents $W_P$, then
$$\abs{Cross^+_P(W_P,C)} = \abs{Cross^+_{G^\infty}(W_{G^\infty},C)}\text{ and }\abs{Cross^-_P(W_P,C)} = \abs{Cross^-_{G^\infty}(W_{G^\infty},C)}.$$ 
\end{claim}

\begin{proof}
We will only consider positive crossings, the proof for negative crossings is similar.
By taking the quotient of a sub-walk of $W_{G^\infty}$, we see that each crossing in ${G^\infty}$ is also present in $P$. Thus $\abs{Cross^+_P(W_P,C)} \supseteq \abs{Cross^+_{G^\infty}(W_{G^\infty},C)}$. Now take a crossing of $C$ by $W_P$ in $P$, it is a sub-walk of $W_P$. Thus if we take the corresponding sub-walk in $W_{G^\infty}$, we get a crossing in ${G^\infty}$. Thus the two sets are equal.
\end{proof}

One of our main results is the following proposition.

\begin{proposition}
\label{prop:sum-cross}
$$ \sum_{W \in \mathcal{W}_a} \abs{Cross^+_P(W,C)} + \abs{Cross^-_P(W,C)} \equiv \setcond{uv \in C}{\text{$uv$ has color $ac$ or $ad$}} \pmod 2.$$
\end{proposition}

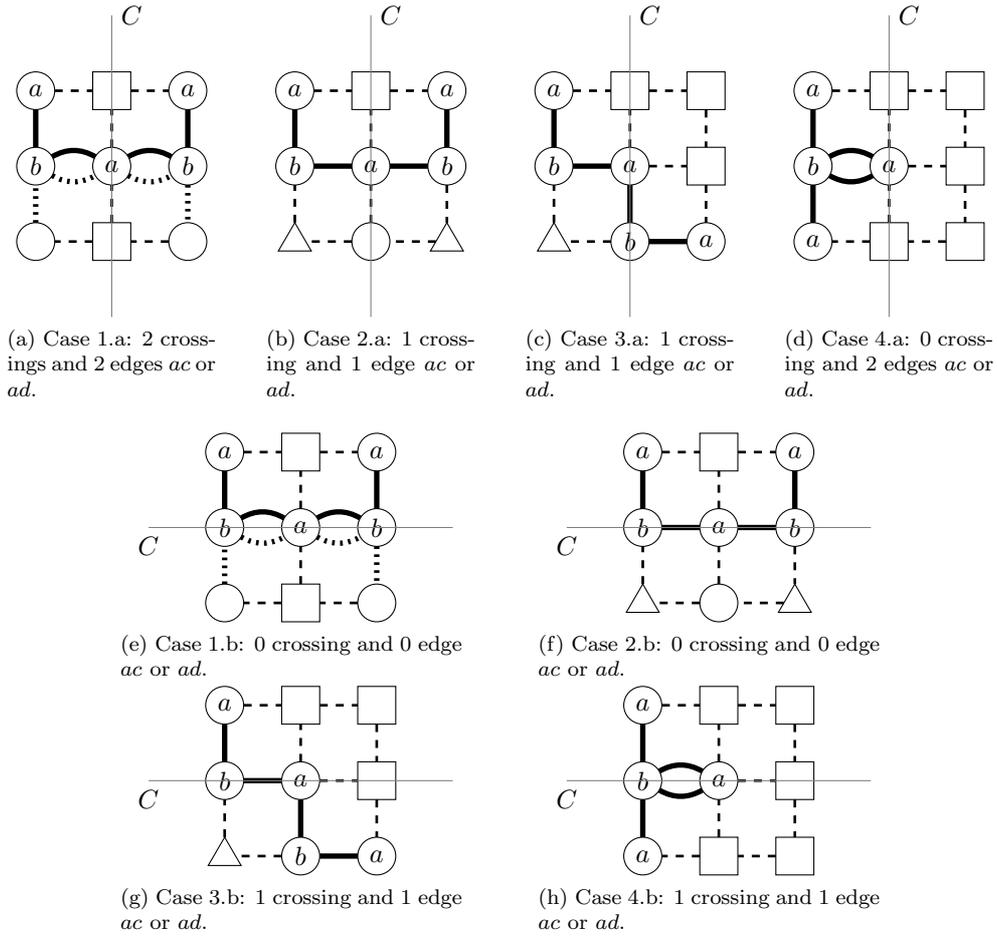
\begin{figure}[p]
\centering
\subfloat[Case 1.a: 2 crossings and 2 edges $ac$ or $ad$.\label{fig:crossing-border-0a}]{
\begin{tikzpicture}
\tikzstylemacro

\node[v] (x) at (0,0) {$a$};
\node[v, rectangle] (w) at (0,1) {};

\node[v] (1) at (1,1) {$a$};
\node[v] (2) at (1,-1) {};
\node[v] (3) at (-1,-1) {};

\node[v] (i1) at (1,0) {$b$};
\node[v, rectangle] (i2) at (0,-1) {};
\node[v] (i3) at (-1,0) {$b$};

\node[v] (i4) at (-1,1) {$a$};

\draw[b,dashed] (x) -- (w) -- (1);
\draw[b,dashed] (2) -- (i2) -- (3);
\draw[b,dashed] (i4) -- (w);
\draw[b,dashed] (i2) -- (x);

\draw[b, line width=2pt] (i1) -- (1);
\draw[b, line width=2pt] (i4) -- (i3);
\draw (x) edge[b, line width=2pt,bend left](i1);
\draw (i3) edge[b, line width=2pt, bend left] (x);

\draw[b, dotted, line width=2pt] (i1) -- (2);
\draw[b, dotted,line width=2pt] (3) -- (i3);
\draw (x) edge[b, dotted, line width=2pt,bend right](i1);
\draw (i3) edge[b, dotted, line width=2pt, bend right] (x);

\draw[gray] (0,2) -- (0,-2);
\coordinate[label=right:$C$] () at (0,2) {};
\end{tikzpicture}
}
\hskip 0.66cm
\subfloat[Case 2.a: 1 crossing and 1 edge $ac$ or $ad$.\label{fig:crossing-border-1a}]{
\begin{tikzpicture}
\tikzstylemacro

\node[v] (x) at (0,0) {$a$};
\node[v, rectangle] (w) at (0,1) {};

\node[v] (1) at (1,1) {$a$};
\node[v, regular polygon,regular polygon sides=3] (2) at (1,-1) {};
\node[v, regular polygon,regular polygon sides=3] (3) at (-1,-1) {};

\node[v] (i1) at (1,0) {$b$};
\node[v] (i2) at (0,-1) {};
\node[v] (i3) at (-1,0) {$b$};

\node[v] (i4) at (-1,1) {$a$};

\draw[b,dashed] (x) -- (w) -- (1);
\draw[b,dashed] (i1) -- (2) -- (i2) -- (3) -- (i3);
\draw[b,dashed] (i4) -- (w);
\draw[b,dashed] (i2) -- (x);
\draw[b, line width=2pt] (x) -- (i1) -- (1);
\draw[b, line width=2pt] (i4) -- (i3) -- (x);

\draw[gray] (0,2) -- (0,-2);
\coordinate[label=right:$C$] () at (0,2) {};
\end{tikzpicture}
}
\hskip 0.66cm
\subfloat[Case 3.a: 1 crossing and 1 edge $ac$ or $ad$.\label{fig:crossing-border-2a}]{
\begin{tikzpicture}
\tikzstylemacro

\node[v] (x) at (0,0) {$a$};
\node[v, rectangle] (w) at (0,1) {};

\node[v, rectangle] (1) at (1,1) {};
\node[v] (2) at (1,-1) {$a$};
\node[v, regular polygon,regular polygon sides=3] (3) at (-1,-1) {};

\node[v, rectangle] (i1) at (1,0) {};
\node[v] (i2) at (0,-1) {$b$};
\node[v] (i3) at (-1,0) {$b$};

\node[v] (i4) at (-1,1) {$a$};

\draw[b,dashed] (i3) -- (x) -- (w) -- (1) -- (i1) -- (2);
\draw[b,dashed] (i2) -- (3) -- (i3) -- (i4) -- (w);
\draw[b,dashed] (x) -- (i1);

\draw[b, line width=2pt] (x) -- (i2) -- (2);
\draw[b, line width=2pt] (i4) -- (i3) -- (x);

\draw[gray] (0,2) -- (0,-2);
\coordinate[label=right:$C$] () at (0,2) {};
\end{tikzpicture}
}
\hskip 0.66cm
\subfloat[Case 4.a: 0 crossing and 2 edges $ac$ or $ad$.\label{fig:crossing-border-3a}]{
\begin{tikzpicture}
\tikzstylemacro

\node[v] (x) at (0,0) {$a$};
\node[v, rectangle] (w) at (0,1) {};

\node[v, rectangle] (1) at (1,1) {};
\node[v, rectangle] (2) at (1,-1) {};
\node[v] (3) at (-1,-1) {$a$};

\node[v, rectangle] (i1) at (1,0) {};
\node[v, rectangle] (i2) at (0,-1) {};
\node[v] (i3) at (-1,0) {$b$};

\node[v] (i4) at (-1,1) {$a$};

\draw[b,dashed]  (x) -- (w) -- (1) -- (i1) -- (2) -- (i2);
\draw[b,dashed] (i2) -- (3) -- (i3) -- (i4) -- (w);
\draw[b,dashed] (i2) -- (x) -- (i1);

\draw[b, line width=2pt] (i4) -- (i3) -- (3);
\draw (x) edge[b, line width=2pt, bend left] (i3) edge[b, line width=2pt, bend right] (i3) ;

\draw[gray] (0,2) -- (0,-2);
\coordinate[label=right:$C$] () at (0,2) {};
\end{tikzpicture}
}

\subfloat[Case 1.b: 0 crossing and 0 edge $ac$ or $ad$.\label{fig:crossing-border-0b}]{
\begin{tikzpicture}
\tikzstylemacro

\node[v] (x) at (0,0) {$a$};
\node[v, rectangle] (w) at (0,1) {};

\node[v] (1) at (1,1) {$a$};
\node[v] (2) at (1,-1) {};
\node[v] (3) at (-1,-1) {};

\node[v] (i1) at (1,0) {$b$};
\node[v, rectangle] (i2) at (0,-1) {};
\node[v] (i3) at (-1,0) {$b$};

\node[v] (i4) at (-1,1) {$a$};

\draw[b,dashed] (x) -- (w) -- (1);
\draw[b,dashed] (2) -- (i2) -- (3);
\draw[b,dashed] (i4) -- (w);
\draw[b,dashed] (i2) -- (x);

\draw[b, line width=2pt] (i1) -- (1);
\draw[b, line width=2pt] (i4) -- (i3);
\draw (x) edge[b, line width=2pt,bend left](i1);
\draw (i3) edge[b, line width=2pt, bend left] (x);

\draw[b, dotted, line width=2pt] (i1) -- (2);
\draw[b, dotted,line width=2pt] (3) -- (i3);
\draw (x) edge[b, dotted, line width=2pt,bend right](i1);
\draw (i3) edge[b, dotted, line width=2pt, bend right] (x);

\draw[gray] (-2,0) -- (2,0);
\coordinate[label=below:$C$] () at (-2,0) {};
\end{tikzpicture}
}
\hskip 1cm
\subfloat[Case 2.b: 0 crossing and 0 edge $ac$ or $ad$.\label{fig:crossing-border-1b}]{
\begin{tikzpicture}
\tikzstylemacro

\node[v] (x) at (0,0) {$a$};
\node[v, rectangle] (w) at (0,1) {};

\node[v] (1) at (1,1) {$a$};
\node[v, regular polygon,regular polygon sides=3] (2) at (1,-1) {};
\node[v, regular polygon,regular polygon sides=3] (3) at (-1,-1) {};

\node[v] (i1) at (1,0) {$b$};
\node[v] (i2) at (0,-1) {};
\node[v] (i3) at (-1,0) {$b$};

\node[v] (i4) at (-1,1) {$a$};

\draw[b,dashed] (x) -- (w) -- (1);
\draw[b,dashed] (i1) -- (2) -- (i2) -- (3) -- (i3);
\draw[b,dashed] (i4) -- (w);
\draw[b,dashed] (i2) -- (x);
\draw[b, line width=2pt] (x) -- (i1) -- (1);
\draw[b, line width=2pt] (i4) -- (i3) -- (x);

\draw[gray] (-2,0) -- (2,0);
\coordinate[label=below:$C$] () at (-2,0) {};
\end{tikzpicture}
}
\hskip 1cm
\subfloat[Case 3.b: 1 crossing and 1 edge $ac$ or $ad$.\label{fig:crossing-border-2b}]{
\begin{tikzpicture}
\tikzstylemacro

\node[v] (x) at (0,0) {$a$};
\node[v, rectangle] (w) at (0,1) {};

\node[v, rectangle] (1) at (1,1) {};
\node[v] (2) at (1,-1) {$a$};
\node[v, regular polygon,regular polygon sides=3] (3) at (-1,-1) {};

\node[v, rectangle] (i1) at (1,0) {};
\node[v] (i2) at (0,-1) {$b$};
\node[v] (i3) at (-1,0) {$b$};

\node[v] (i4) at (-1,1) {$a$};

\draw[b,dashed] (i3) -- (x) -- (w) -- (1) -- (i1) -- (2);
\draw[b,dashed] (i2) -- (3) -- (i3) -- (i4) -- (w);
\draw[b,dashed] (x) -- (i1);

\draw[b, line width=2pt] (x) -- (i2) -- (2);
\draw[b, line width=2pt] (i4) -- (i3) -- (x);

\draw[gray] (-2,0) -- (2,0);
\coordinate[label=below:$C$] () at (-2,0) {};
\end{tikzpicture}
}
\hskip 1cm
\subfloat[Case 4.b: 1 crossing and 1 edge $ac$ or $ad$.\label{fig:crossing-border-3b}]{
\begin{tikzpicture}
\tikzstylemacro

\node[v] (x) at (0,0) {$a$};
\node[v, rectangle] (w) at (0,1) {};

\node[v, rectangle] (1) at (1,1) {};
\node[v, rectangle] (2) at (1,-1) {};
\node[v] (3) at (-1,-1) {$a$};

\node[v, rectangle] (i1) at (1,0) {};
\node[v, rectangle] (i2) at (0,-1) {};
\node[v] (i3) at (-1,0) {$b$};

\node[v] (i4) at (-1,1) {$a$};

\draw[b,dashed]  (x) -- (w) -- (1) -- (i1) -- (2) -- (i2);
\draw[b,dashed] (i2) -- (3) -- (i3) -- (i4) -- (w);
\draw[b,dashed] (i2) -- (x) -- (i1);

\draw[b, line width=2pt] (i4) -- (i3) -- (3);
\draw (x) edge[b, line width=2pt, bend left] (i3) edge[b, line width=2pt, bend right] (i3) ;

\draw[gray] (-2,0) -- (2,0);
\coordinate[label=below:$C$] () at (-2,0) {};
\end{tikzpicture}
}
\caption{All cases for the central vertex to belong to a closed walk $W \in \mathcal W_a$. We use the same notation as in Figure~\ref{fig:start-border}. The dotted lines in sub-figures~\ref{fig:crossing-border-0a} and~\ref{fig:crossing-border-0b} are a second passage in the central vertex by a walk in $\mathcal W_a$ (possibly the same as the bold line). The edges $ac$ or $ad$ are the edges between a circular vertex and a square vertex. For each case we count the number of crossings of the drawn walks and the number of edges of color $ac$ or $ad$ incident to the central vertex and belonging to the cycle $C$.}
\label{fig:crossing-border}
\end{figure}

\begin{proof}
Take a vertex $x$ of $C$ in $(P,\pi')$ colored $a$. If the region of $x$ is  $\set{x}$, then $x$ has two incident edges colored $ac$ or $ad$ and $x$ is not contained in any crossing as it does not belong to a walk in $\mathcal W_a$ by definition. Thus we can ignore them.

If $x$ has at least one incident edge colored $ac$ or $ad$, then it belongs to some border colored $a$.

Now take a vertex $x$ of color $a$ in some walk $W \in \mathcal W_a$. Depending on the size of $BN(x)$ there are one or two occurrences of $x$ in $\mathcal W_a$ by Claim~\ref{claim:properties-border}. Up to rotation we can suppose that we have $v_{i-1,j+1},v_{i-1,j},v_{i,j}=x$ as a sub-walk of $W$. Depending on the orientation of $C$ (vertical or horizontal), for each sub-case, we must consider the two orientations. For one orientation there are four sub-cases: $\abs{BN(x)} = 4$, $\abs{BN(x)} = 2$ and we chose $v_{i+1,j+1}$ during the construction, $\abs{BN(x)} = 2$ and we chose $v_{i+1,j-1}$ during the construction or $\abs{BN(x)} = 2$ and we chose $v_{i-1,j-1}$ during the construction. All the sub-cases are depicted in Figure~\ref{fig:crossing-border}. In each case the number of crossings for the sub-walks considered is equal, modulo $2$, to the number of edges colored $ab$ of $x$ in $C$.

Now note that no vertices of color $b$ in $W \in W_a$ has both neighbors in the same layer. Thus a crossing of $C$ by $W$ always contains a vertex colored $a$ of $C$. Thus all crossings are counted in the above case analysis.

Since for each edge colored $ac$ or $ad$, the vertex colored $a$ has a neighbor not in its region, it is on some border and thus we counted those edges in the case analysis or when we treated the case of the region of size one.

Thus the number of edges colored $ac$ or $ad$ in $C$ is equal to the sum of the number of crossings of $C$ by walks in $\mathcal W_a$ modulo $2$.
\end{proof}

\begin{claim}
\label{claim:ac-ad-ab}
The number of edges colored $ac$ or $ad$ in $C$ is equal to the number of edges of $C$ colored $ab$ modulo $2$.
\end{claim}

\begin{proof}
Let us call $E_{ac}$ (\resp $E_{ad}$, \resp $E_{ab}$) the set of edges colored $ac$ (\resp $ad$, \resp $ab$). 
Since a vertex of color $a$ has two incident edges in $C$, we have:
\begin{align*}
\abs{(E_{ac} \cup E_{ad}) \cap C} &\equiv \sum_{x\in C \text{ of color a}} \abs{(E_{ac} \cup E_{ad}) \cap N(x) \cap C} \pmod 2\\
 &\equiv \sum_{x\in C \text{ of color a}} \deg_C(x) - \abs{E_{ab} \cap N(x) \cap C} \pmod 2 \\
 &\equiv \sum_{x\in C \text{ of color a}} 2 - \abs{E_{ab} \cap N(x) \cap C} \pmod 2 \\
 &\equiv \sum_{x\in C \text{ of color a}} \abs{E_{ab} \cap N(x) \cap C} \pmod 2 \\
 &\equiv \abs{E_{ab} \cap C} \pmod 2.
\end{align*}
\end{proof}

\subsection{Crossings and turns}
\label{sec:crossings-turns}

In this section, we will suppose that $C$ is the vertical cycle of $(P,\pi')$ equal to $\setcond{\overline{v_{x,y}}}{x = x_0 + n\ell,\ n \in \NN}$ for $\ell = 2p+1$. We identify $C$ on $(P,\pi')$ and the set $\setcond{v_{x,y}}{\overline{v_{x,y}} \in C}$ of vertices of ${G^\infty}$. All what is defined below also works for a horizontal cycle with $\ell =q$. Here we want to connect the number of crossings of the previous section with the number of turns of Section~\ref{sec:turns}.

\begin{definition}
Let $v_{x,y}$ be a vertex of ${G^\infty}$. We define the function $g_C$  as follows:
$$ g_C(v_{x,y}) = \floor{\frac{x-x_0}{\ell}}.$$

For a walk $W_{G^\infty}$ in ${G^\infty}$ starting at $s_0$ and finishing at $s_n$, we define $f_C$ as follows:
$$ f_C(W_{G^\infty}) = g_C(s_n) -g_C(s_0).$$
\end{definition}

\begin{claim}
\label{claim:fc-cross}
For a walk $W_{G^\infty}$ of ${G^\infty}$ representing a closed walk $W_P$ of $P$ with starting point $v_{x,y} \notin C$:
$$ f_C(W_{G^\infty}) \equiv \abs{Cross^+_{G^\infty}(W_{G^\infty},C)} - \abs{Cross^-_{G^\infty}(W_{G^\infty},C)}\pmod 2 .$$
\end{claim}

\begin{proof}
Suppose $W_{G^\infty} = (s_i)_{i \in \set{0,\dots,n}}$, by assumption $s_0 \notin C$.
This ensures that all crossings of $C$ by $W_{G^\infty}$ are sub-walks that do not go through the end of $W_{G^\infty}$ and go back at the beginning. Take a crossing $t_0,\dots,t_k$. We have $g_C(t_k) - g_C(t_0) = 1$ if the crossing is positive and $g_C(t_k) - g_C(t_0) = -1$ if it is negative.

Now we just have to show that the other sub-walks of $W_{G^\infty}$ do not contribute to $f_C(W_{G^\infty})$.  We can write $W_{G^\infty} = W_0,W^{cross}_0,W_1,\dots,W^{cross}_{k-1},W_k$ for some integer $k$ where each $W^{cross}_i$ is a crossing and the other sub-walks are not. 
Note that: $$f_C(W_{G^\infty}) = \sum\limits_{i \in \set{0,\dots,k}} f_C(W_i) + \sum\limits_{i \in \set{0,\dots,k-1}} f_C(W^{cross}_i).$$

If for all $i \in \set{0,\dots,k}$, $f_C(W_i) = 0$, we have our result.
Since the endpoints of the $W_i$'s are the same as the starting points of the crossings, we know that they do not belong to $C$. The same is true for the starting points. Then, for the starting point $v_{x,y} $ and  the endpoint $v_{z,t}$, we have $x,z \in \set{x_0 + n \ell + 1, \dots,  x_0 + n \ell + l -1}$ for some $n$. But in all cases the value of $g_C$ is $n$. Thus $f_C(W_i) = 0$. This concludes the proof.
\end{proof}

\begin{claim}
\label{claim:fc-tau}
For a closed walk $W_P$ in $P$ and $W_{G^\infty}$ a representation of $W_P$ on ${G^\infty}$:

$$ f_C(W_{G^\infty}) \equiv \tau_x(W_P) \pmod 2 .$$
\end{claim}

\begin{proof}
Suppose that $W_{G^\infty}$ starts at $v_{x,y}$ and ends at $v_{z,t}$. Note that $z = x + n\ell$ for some $n \in \ZZ$.
We have:
\begin{align*}
 \tau_x(W_P) &\equiv  \abs{\frac{z-x}{\ell}} \pmod 2\\
 &\equiv n \pmod 2,\\
\end{align*}
while:
\begin{align*}
 f_C(W_{G^\infty}) &\equiv  \floor{\frac{z-x_0}{\ell}} - \floor{\frac{x-x_0}{\ell}} \pmod 2\\
 &\equiv n + \floor{\frac{x-x_0}{\ell}} - \floor{\frac{x-x_0}{\ell}} \pmod 2\\
 & \equiv n \pmod 2.
\end{align*}
\end{proof}

\subsection{End of the proof}
\label{sec:proof-ending}

We can now prove Lemma~\ref{lem:ProdCyclePas4}.

\begin{proof}
Note that by shifting the indices, we can suppose that the starting vertex  of each $W_P \in \mathcal W_a$ does not belong to $C$.
By using Claim~\ref{claim:ac-ad-ab}, Proposition~\ref{prop:sum-cross}, Claim~\ref{claim:cross-quotient}, Claim~\ref{claim:fc-cross} and Claim~\ref{claim:fc-tau}, in this order, we get:
\begin{align*}
\abs{E_{ab} \cap C} & \equiv \sum_{W_P \in \mathcal{W}_a} \abs{Cross^+_P(W_P,C)} + \abs{Cross^-_P(W_P,C)} \pmod 2\\
& \equiv \sum_{\substack{W_P \in \mathcal{W}_a \\ W_{G^\infty}\text{ represents $W_P$ and}\\ \text{its starting point $\notin C$}}} \abs{Cross^+_{G^\infty}(W_{G^\infty},C)} - \abs{Cross^-_{G^\infty}(W_{G^\infty},C)} \pmod 2\\
& \equiv \sum_{\substack{W_P \in \mathcal{W}_a \\ W_{G^\infty}\text{ represents $W_P$ and}\\ \text{its starting point $\notin C$}}} f_C(W_{G^\infty}) \pmod 2\\
& \equiv \sum_{W_P \in \mathcal{W}_a} \tau_x(W_P) \pmod 2.\\
\end{align*}

By the choice of $C$ in the previous subsection, $C = UC_q$ and thus $\abs{E_{ab} \cap C} \equiv 1 \pmod 2$. Therefore:
$$ 1 \equiv \sum_{W_P \in \mathcal{W}_a} \tau_x(W_P) \pmod 2.$$

By taking $C = BC_{2p+1}$, a horizontal cycle, we obtain:
\begin{align*}
0 \equiv \abs{E_{ab} \cap C}  \equiv \sum_{W \in \mathcal{W}_a} \tau_y(W) \pmod 2.
\end{align*}

Recall that Proposition~\ref{prop:sum-tau} states that:
$$ 0 \equiv q \tau_y(W) + \tau_x(W) \pmod 2.$$

Thus:
$$ 0 \equiv q \times 0 + 1 \pmod 2.$$

This is a contradiction.
\end{proof}

This concludes the proof of Lemma~\ref{lem:ProdCyclePas4}.

\section{Conclusion}
\label{sec:cartesian:conclu}

To conclude, in this paper, we showed a number of results on Cartesian products of signed graphs.
We proved some algebraic properties: Theorem \ref{thm:cartesianHomCompatibility}, Theorem \ref{thm:primefactortheorem} and Theorem \ref{thm:cancellation-property}. We also presented an optimal algorithm to decompose a signed graph into its factors in time $O(m)$. 

Finally, we computed the chromatic number of Cartesian products: Cartesian products of any graph  by a signed forest, Cartesian  products of signed paths, signed graphs with underlying graph $P_n \ssquare P_m$, Cartesian products of some signed complete graphs and Cartesian products of signed cycles. We also presented a tool called an $s$-redundant set that helped  to compute chromatic numbers of signed graphs.
 It would be interesting to determine the exact 
 chromatic number of a signed grid. In this paper, we only presented an upper bound and the question whether $5$ or $6$ is the best upper bound is still open. It would also be interesting to compute the chromatic number of more  Cartesian products.

\section{Acknowledgements}

We would like to thank Hervé Hocquard and Éric Sopena for their helpful comments through the making of this paper. We would also like to thank the reviewers of our submission to CALDAM 2020 for their comments, especially Reviewer 2 of our submission to CALDAM 2020 for pointing us to the techniques of \cite{Imrich2018} which improved our algorithm. This work is partially supported by the ANR project HOSIGRA (ANR-17-CE40-0022) and the IFCAM project ``Applications of graph homomorphisms'' (MA/IFCAM/18/39).

\bibliographystyle{plain}
\bibliography{biblio}

%
%
%

\end{document}